\documentclass[12pt, leqno]{amsart}
\usepackage{latexsym}
\usepackage{amssymb}
\usepackage{amsmath}
\usepackage[english]{babel}
\usepackage[colorlinks=true, pdfstartview=FitV, linkcolor=red, citecolor=red, urlcolor=blue, backref=page]{hyperref} 
\usepackage{tikz, tikz-3dplot}
\usepackage{tkz-euclide}
\usetikzlibrary{decorations.fractals}
\usetikzlibrary{decorations.footprints}
\usepackage{wrapfig}

\usepackage[top=1.25in, bottom=1.25in, left=1.25in, right=1.25in]{geometry}

\newtheorem{theorem}{Theorem}

\newtheorem{lemma}[theorem]{Lemma}

\newtheorem{remark}[theorem]{Remark}

\newtheorem{mytheorem}{Theorem}

\newcommand{\rn}[1]{\mathbb{R}^{#1}}

\newcommand{\re}{ \mathbb{R}}

\newcommand{\beq}{\begin{equation}}
\newcommand{\bea}[1]{\begin{array}{#1} }
\newcommand{\eeq}{ \end{equation}}
\newcommand{\ea}{ \end{array}}
\newcommand{\ep}{\epsilon}
\newcommand{\es}{\emptyset}

\newcommand{\al}{\alpha}
\newcommand{\ga}{\gamma}

\newcommand{\de}{\delta}
\newcommand{\ds}{\displaystyle}
\newcommand{\ts}{\textstyle}

\newcommand{\ran}{\rangle}
\newcommand{\lan}{\langle}

\newcommand{\la}{\lambda}
\newcommand{\La}{\Lambda}
\newcommand{\ar}{\partial}
\newcommand{\si}{\sigma}

\newcommand{\om}{\omega}
\newcommand{\Om}{\Omega}
\newcommand{\be}{\beta}

\newcommand{\ph}{\phi}
\newcommand{\he}{\theta}

\newcommand{\Ph}{\Phi}
\newcommand{\hs}[1]{\mbox{$ \hspace{#1}$}}

\newcommand{\sem}{\setminus}
\newcommand{\ze}{\zeta}

\newcommand{\ti}{\tilde}

\usepackage{filecontents}

\renewcommand*{\backref}[1]{}
\renewcommand*{\backrefalt}[4]{%
 \ifcase #1 (Not cited.)%
   \or        (Cited on page~#2.)%
    \else      (Cited on pages~#2.)%
    \fi}  

\begin{document} 

\title{On a  Theorem of  Wolff  Revisited} 

\dedicatory{Dedicated to the memory of Thomas Wolff}

\author[M. Akman]{Murat Akman}
\address{{\bf Murat Akman}\\
Department of Mathematical Sciences, University of Essex\\
Wivenhoe Park, Colchester, Essex CO4 3SQ, UK
} \email{murat.akman@essex.ac.uk}

\author[J. Lewis]{John Lewis}
\address{{\bf John Lewis} \\ Department of Mathematics \\ University of Kentucky \\ Lexington, Kentucky, 40506}
\email{johnl@uky.edu}

\author[A. Vogel]{Andrew Vogel}
\address{{\bf Andrew Vogel}\\ Department of Mathematics, Syracuse University \\  Syracuse, New York 13244}
\email{alvogel@syracuse.edu}


\begin{abstract}
  We   study  $p$-harmonic functions,  $ 1 < p\neq 2 < \infty$, in     $ \rn{2}_+  =  \{ z = x + i y : y > 0,  - \infty < x < \infty \}  $  and   
  $ B ( 0, 1 )   =   \{ z :   |z| < 1 \}$.  We first  show    
  for  fixed   $ p$, $1 <  p\neq 2  <   \infty$, and for all large integers  $N\geq N_0$  that   there  exists  $p$-harmonic function,
  $ V = V ( r  e^{i\he} )$,   which is $  2\pi/N  $ 
  periodic in the $ \he $ variable,    and  Lipschitz continuous on  $  \ar  B (0, 1)$ with   Lipschitz norm
   $\leq  c N$ on $  \ar B ( 0, 1 )$ satisfying    $V(0)=0$ and    $    c^{-1}  \leq    
  \int_{-\pi}^{\pi} V  (  e^{i\he} )  d \he    \leq  c$.  
 In case $2<p<\infty $ we  give   a   more or less explicit  example of  $V$ and our work  is an extension of a result of  Wolff in \cite[Lemma 1]{W}  on    
$  \rn{2}_+ $  to   $ B (0, 1).$  Using our first result, we extend  the work of Wolff in \cite{W} on failure of Fatou type theorems  for   $ \rn{2}_+ $  to $ B (0, 1)$ for $p$-harmonic functions, $1< p\neq 2<\infty$.  Finally, we  also  outline   the    modifications  needed  for extending the   work of  Llorente,  Manfredi,  and Wu in  \cite{LMW}
 regarding failure of subadditivity of  $p$-harmonic measure on $ \ar \rn{2}_+   $   to  $\ar  B (0, 1)$.    
 \end{abstract}

\maketitle
\setcounter{tocdepth}{1}
\tableofcontents

\section{Introduction}  
\label{sec1}    
Throughout this  paper we mix  complex and real notation,  so   $ z = x+iy$ and  
$\bar z =  x - i y$ whenever $x, y \in \re$ where $i =  \sqrt{-1}$.   Moreover, we let $ \rn{2}_+
= \{   z =   x + i y : y > 0 \}  $   and  $  B ( z_0, \rho  )  =  \{ z : | z - z_0 |   <    \rho \} $  
whenever  $ z_0 \in \rn{2}$ and $\rho  > 0. $   We  consider  for    fixed $ p$, $1 < p\neq 2 < \infty$,   
weak solutions $u$   (called  $p$-harmonic functions)    to   $p$-Laplace equation
\begin{align}
 \label{1.1} 
\mathcal{L}_p u   :=   \nabla  \cdot  ( | \nabla u |^{p-2}   \nabla u )  =  0 
\end{align}
 on  $ B  (0, 1 ) $ or  $  \rn{2}_+ $ (see  section  \ref{sec2} for  definition of  a  $p$-harmonic function).  
 In  \eqref{1.1}, $  \nabla  u$  denotes the  gradient of $u$  and  $  \nabla \cdot $  
 denotes  the divergence operator.       In   1984  Wolff brilliantly  used ideas  from  harmonic analysis and  PDE  to  prove Fatou theorem fails for $p$-harmonic functions for $2<p<\infty$. 
\begin{theorem}[{\cite[Theorem 1]{W}}]
\label{thm1.1} 
 If  $ p  > 2,  $   then there exist   bounded  weak  solutions of  $ \mathcal{L}_p \hat u  = 0 $ in 
  $  \rn{2}_+  $  such that $\{ x \in \re  : {\ds  \lim_{ y \to 0 } \hat  u ( x + i  y ) }$ exists\}   
  has  Lebesgue measure  zero.  Also    there exist positive  bounded weak solutions of $ \mathcal{L}_p \hat v = 0$  such that  $ 
\{ x \in \re  :  {\ds \limsup_{y \to 0 } \,   \hat v   (x + i  y)  > 0 \}}  $  has Lebesgue measure $0$.    
\end{theorem}    

The  key to his proof  and  the only obstacle in extending  Theorem \ref{thm1.1} to $ 1 < p < \infty $    
was the  validity of  the  following  theorem for $ 1 < p < 2, $   stated  as  Lemma 1 in  \cite{W}.  
\begin{theorem}[{\cite[Lemma 1]{W}}]
\label{thm1.2}   
 If  $ p > 2$  there exists a  bounded  Lipschitz function $ \Ph $  on  
the closure of $  \rn{2}_+ $  with   $  \Ph ( z + 1 )  =  \Ph ( z )$ for $z \in \rn{2}_+$,   $\mathcal{L}_p  \Phi  = 0   $ 
weakly  on $  \rn{2}_+$, $   \int_{ (0,1) \times (0, \infty) }  | \nabla  \Ph |^p   \, dx dy   <  \infty,  $  
and    
\begin{align}
\label{1.2} 
  {\ds   \lim_{y \to \infty} \Phi ( x + i  y ) = 0}\, \, \mbox{for}\, \, x \in \re,   \quad  \mbox{but}\quad     \int_0^1  \,  \Ph ( x )  dx   \neq 0.   
\end{align}
\end{theorem}  
 Theorem \ref{thm1.2} was  later  proved for $ 1 < p < 2, $    by the second author of this article
  in   \cite{L1}  (so  Theorem \ref{thm1.1}  is  valid   for  $  1 < p\neq 2 <  \infty). $   
  Wolff remarks above the statement of   his Lemma 1,  that ``Theorem \ref{1.1}     
  should generalize to  other domains  but  the  arguments are easiest in a 
  half space since  $ \mathcal{ L}_p $   behaves  nicely under  Euclidean operations''.

 In fact Wolff  makes extensive use of  the fact that $ \Ph ( N  z + z_0 )$, $z=x+iy\in \rn{2}_+$,  $ N $  a  
 positive integer,  $ z_0   \in \rn{2}_+,   $ is  $p$-harmonic in $ \rn{2}_+ ,     $ 
 and  $ 1/N $  periodic in $x$,  with  Lipschitz norm  $ \approx N $  on $ \re  = \ar \rn {2}_+$.    
Also he  used     functional analysis-PDE  type arguments,  involving the  
Fredholm alternative and perturbation of  certain $p$-harmonic functions (when $2<p<\infty$)
to get $\Phi $ satisfying \eqref{1.2}.     
 
In this paper we first give  in   Lemma  \ref{lem3.1},   a  hands on example of  a  $ \Ph $  for which  
 Theorem \ref{thm1.2} is  valid.    We then use this  example  and  basic properties of  
 $p$-harmonic functions to  give  a  more or less explicit construction   of   $  V  =  V ( \cdot, N, p) $  
 for $2<p<\infty$   in the following theorem.  
\begin{mytheorem} 
\label{thmA}  
Given $ p$,  $1 < p\neq 2 <  \infty$, there  exist  $N_0$  and  a constant $ c_1 \geq 1$, 
all depending only on $ p$,   such that if   $ N  \geq N_0 $ is a positive integer,  then  
there is a  $ p $-harmonic function $  V $  in  $ B ( 0, 1) $ with continuous boundary values satisfying  
\begin{align} 
\label{1.3} 
\begin{split}
&(a) \hs{.2in} - c_1  \leq   V ( t e^{i\he} ) =  V ( t e^{i ( \he + 2\pi/N) } )    \leq c_1 \quad \mbox{for} \, \,   0 \leq t  \leq  1 \, \, \mbox{and}\, \,  \he  \in \re, \\
 &     (b)  \hs{.2in}   { \ds \int_{ B(0, 1 )} }    |\nabla V|^p  dx dy \, \leq   \, c_1  \, N^{p - 1}, \\
 &    (c)  \hs{.2in}     V ( 0 ) = 0  \mbox{ and }  c_1  \,    
 { \ds   \int_{ -  \pi}^{ \pi}    V ( e^{i\he} )   d \he}    \geq  1,  \\
 & (d) \hs{.2in}   V|_{\ar B ( 0, 1 )}   \mbox{ is Lipschitz   with norm }
  \leq c_1  N. 
\end{split}
\end{align} 
  \end{mytheorem} 

We  were  not able  to find  a  more or less  explicit  example  for which    Theorem \ref{thmA}  
holds  when  $ 1 < p < 2. $  Instead  for  $ 1 < p\neq 2 < \infty, $  we  also  use   a  finesse type argument  
 to  eventually    obtain   Theorem \ref{thmA}   from  the perturbation method used in proving  
 Theorem \ref{thm1.2}  and a    limiting type argument.       In this proof  of  Theorem \ref{thmA}  
 we also interpret rather loosely the phrase  ``$ c_1 $ depends only on $p$''.
 However  constants   will  always be    independent of  $ N  \geq N_0 . $  
 We shall make heavy use of Wolff's  arguments in proving  Theorem \ref{thm1.2},   
 as well as  arguments of  
Varpanen  in   \cite{V},   who    adapted  Wolff's  perturbation  argument    for  constructing solutions
 to a  linearized $p$-harmonic  periodic   equation in $ \rn{2}_+$   to  certain  periodic $p$-harmonic 
 functions in the  $ \he  $  variable,   defined on $ B (0, 1). $   In  section \ref{sec4}    we   use 
 Theorem  \ref{thmA} and   modest changes  in  Wolff's  argument   to  obtain the following   
 analogue  of   Theorem  \ref{thm1.1}.         
\begin{mytheorem}  
\label{thmB}  
If  $2<p<\infty$,   then there exist bounded  weak  solutions of  $ \mathcal{L}_p \hat u  = 0 $ in 
  $  B ( 0, 1)  $  such that   $ \{ \he  \in \re  : {\ds  \lim_{ r \to 1 } \hat  u ( r e^{i\he})}$ exists\}   has  Lebesgue measure  zero.  
  Also    there exist  bounded  positive  weak solutions of $ \mathcal{L}_p \hat v = 0$  such that  $ 
\{ \he  \in \re  :  {\ds \limsup_{r \to 1 } \,   \hat v   (r e^{i\he}) > 0 \}}  $  has Lebesgue measure 0.    
\end{mytheorem}    

Next  for fixed $ p  > 1, $  and  $E$  a subset of  $ \ar B ( 0, 1 ), $   let  $ \mathcal{C} (E), $  
denote the  class of  all non-negative  $p$-superharmonic functions  $\ze$  on $ B ( 0, 1) $ 
 (i.e.,  $\mathcal{L}_p \ze  \leq 0$ weakly  in  $  B ( 0, 1 )$)    with   
\begin{align}
\label{1.4} 
{\ds  \liminf_{\stackrel{z \in B(0,1)}{z\to e^{i\he}}}}\, \ze ( z )   \geq 1 \quad  \mbox{for all}\, \,   e^{i\he} \in E.  
\end{align}
Put   $ \om_p ( z_0, E )  =  \inf  \{ \ze (z_0) : \ze \in  \mathcal{C} (E)\} $  when  $ z_0 \in B ( 0, 1). $   Then 
$ \om_p (z_0,  E ) $  is usually  referred to as the  $p$-harmonic measure of  $ E $  relative to $ z_0 $
 and $ B (0, 1 ). $      In   section  \ref{sec5}  we  use  Theorem   \ref{thmA}  and  
 follow  closely   Llorente,  Manfredi, and Wu  in generalizing their work, \cite{LMW},   
 on  $ p $-harmonic measure in  $  \rn{2}_+ $   to  $ B (0, 1 ) . $  We prove  
\begin{mytheorem}  
\label{thmC}  
If  $ 1 < p\neq 2 < \infty $  there exist finitely many sets 
$ E_1, E_2, \dots, E_{\kappa}    \subset  \ar B (0, 1), $  such that  
\begin{align}
 \label{1.5}   
 \om_p  (0, E_k ) = 0, \quad  \om_p ( 0,  \ar B (0, 1 ) \sem E_k ) = 1 \, \, \mbox{for}\, \, 1 \leq k \leq \kappa,  \quad \mbox{and} \quad  \bigcup_{k=1}^{\kappa}    E_k   =   \ar B (0, 1 ).  
 \end{align}
 Furthermore,  $ \ar B ( 0, 1) \sem E_k  $  has  one Lebesgue  measure 0  for  $ 1 \leq k \leq \kappa. $   
\end{mytheorem} 
   
As for the plan of this paper,  in  section \ref{sec2}  we give some  definitions and state some basic properties 
 of  $p$-harmonic functions.  As outlined above,  in  sections  \ref{sec3}, \ref{sec4}, \ref{sec5},     we prove  
 Theorems  \ref{thmA},    \ref{thmB},   \ref{thmC}, respectively.   Finally, in  section \ref{sec6}, we  make  closing remarks.

\section{Basic estimates and definitions for $p$-harmonic functions}
\label{sec2}  
\setcounter{equation}{0} 
 \setcounter{theorem}{0}
 
 In this section we  first introduce some notation, then give some definitions,  and  finally state  some fundamental estimates for 
   $p $-harmonic functions when $ p $ is fixed, $ 1 < p < \infty. $ As in the introduction we  set 
$ B ( z_0, \rho )  = \{ z:\,  | z - z_0 |  < \rho \}$ and $\rn{2}_+   = \{ z = x + i y :\,  y > 0 \}. $     
Concerning constants, unless otherwise stated, in this section, and throughout the paper,
$ c $ will denote a  positive constant  $ \geq 1$, not
necessarily the same at each occurrence, depending only on $ p. $ Sometimes we write $ c = c (p) $
 to indicate this dependence.  Also   $ A  \approx  B $  means  $  A/B $ is bounded above and below 
 by  constants depending only on  $p.$  Let  $ d ( E_1, E_2 )$  denote the distance between the sets 
 $ E_1$ and $E_2. $ For short we write  $ d ( z,  E_2  ) $ for  $ d  ( \{z\}, E_2). $  Let $\mbox{diam}(E)$, 
 $\bar E$, and $\ar E $  denote  the diameter, closure, and boundary  of  $ E$ respectively. 
We also write
$ { \ds \max_{E}   \hat v, \,  \min_{E} \hat v } $ to denote  the
  essential supremum and infimum  of $  \hat v  $ on $ E$
whenever $ E \subset  \rn{2} $ and  $ \hat v $ is defined on $ E$.
  
If $ O  \subset \mathbb R^{2} $ is open and $ 1  \leq  q  \leq  \infty, $ then by   $
W^{1 ,q} ( O ) $ we denote the space of equivalence classes of functions
$ h $ with distributional gradient $ \nabla h, $ both of which are $q$-th power integrable on $ O. $  Let  
 \[
 \| h \|_{1,q} = \| h \|_q +  \| \, | \nabla h | \, \|_{q}
 \]
be the  norm in $ W^{1,q} ( O ) $ where $ \| \cdot \|_q $ is
the usual  Lebesgue $ q $ norm  of functions in the Lebesgue space $ L^q(O).$  Next let $ C^\infty_0 (O )$ be
 the set of infinitely differentiable functions with compact support in $
O $ and let  $ W^{1,q}_0 ( O ) $ be the closure of $ C^\infty_0 ( O ) $
in the norm of $ W^{1,q} ( O  ). $   Let  $ \lan \cdot, \cdot \ran $  denote the standard inner product on $ \rn{2}. $  
Given an  open set 
 $ O  $  and $ 1 < p < \infty, $  we say that $ \hat v  $ is $p$-harmonic in $ O $ provided $ \hat v \in W^ {1,p} (G) $ for each open $ G $ with  $ \bar G \subset O $ and   
	\begin{align}
	\label{2.1}
		\int    | \nabla \hat  v |^{p-2}   \lan    \nabla \hat v,  \nabla \he 
 \ran \, dx dy  = 0 \quad \mbox{whenever} \, \,\he \in W^{1, p}_0 ( G ).
			\end{align} 

	We  say that $  \hat v  $ is a $p$-subsolution  ($p$-supersolution) in $O$ 
	provided    $  \hat v  \in W^{1,p} (G) $ whenever $ G $ is as above and  \eqref{2.1} holds with
$=$ replaced by $\leq$ ($\geq$) whenever $  \theta  \in W^{1,p}_{0} (G )$ with $\theta \geq 0$.  
We begin our statement of lemmas with the following maximum principle.  
\begin{lemma}  
\label{lem2.0}  
Given $1<p<\infty$, if  $  \hat v $ is  a $p$-subsolution and  $ \hat h $ is  a  $p$-supersolution in $ O $ with 
 $ \max ( \hat v -  \hat h, 0 )  \in  W_0^{1, p} (G) , $ 
 whenever $ G $ is an open set with  $  \bar G  \subset O, $   then   $  { \ds \max_O ( \hat v -  \hat h)}  \leq 0 . $ 
\end{lemma} 

\begin{proof}
A proof of this lemma can be found in \cite[Lemma 3.18]{HKM}.
\end{proof}

%
\begin{lemma}
\label{lem2.1}
Given $ p, 1 < p < \infty,   $   let
 $ \hat v $  be  $p$-harmonic in $B (z_0 ,4\rho)$ for some $\rho>0$ and $z_0  \in \rn{2}$. Then  
\begin{align}
 	\label{2.2} 
\begin{split}
&(a)  \hs{.2in}  {\ds  \max_{B (z_0,  \rho/2)}  \, \hat v  -     \min_{B ( z_0, \rho/2 ) }  \hat v     \leq  c \left( \rho^{ p - 2} \,\int_{B ( z_0, \rho)} \, | \nabla \hat v |^{ p } \, dx dy \right)^{1/p}   \, \leq \, c^2  \, (\max_{B ( z_0,
 2 \rho )} \hat  v   -   \min_{B ( z_0, 2 \rho ) }  \hat v  ).} \\
&	\mbox{ Furthermore, there exists $\ti \al  = \ti \al (p)  \in (0,1)$ such that if $ s \leq  \rho$  then} \\ 			
& (b) \hs{.2in}   {\ds  \max_{B(z_0,s) }  \hat v  - \min_{B ( z_0, s) } \hat v  \leq  c \left( \frac{s}{\rho} \right)^{\ti \al} \,  \left(\max_{B (z_0, 2 \rho )}  \, \hat v  -     \min_{B ( z_0, 2\rho ) }  \hat v \right) .} \\  
&(c) \hs{.2in}  \mbox{ If  $ \hat v \geq 0 $ in $ B ( z_0, 4 \rho), $  then }      {\ds  \max_{B(z_0, 2\rho ) }  \hat v  \leq   c  \min_{B ( z_0, 2 \rho ) } \hat v }. 
\end{split}
\end{align}
   \end{lemma}  
\begin{proof}
 For  a proof of  Lemma \ref{lem2.1},   see  chapter 6 in   \cite{HKM}.   Here \eqref{2.2}  $(c)$ is called  Harnack's inequality.  
\end{proof}
\begin{lemma} 
\label{lem2.2} 
Let  $  \Om  =  B ( 0, 1)  $  or   $  \rn{2}_+ $  and $ 1 < p < \infty. $   Let $ z_0   \in  \ar \Om $   and suppose  
 $  \hat v $ is  $p$-harmonic  in   $  \Om \cap  B  ( z_0, 4 \rho  )$ for $0 <  \rho  <\mbox{diam}(\Om)$   
  with $ \hat h  \in  W^{1,p}  ( \Om \cap   B  ( z_0, 4 \rho  ) ) $   and  
  $ \hat v  - \hat h  \in  W^{1,p}_0  ( \Om \cap   B  ( z_0, 4 \rho  ) ). $   
  If    $  \hat h $ is continuous on $  \ar  \Om \cap  B  ( z_0, 4 \rho  )$  
  then $ \hat v $ has a continuous extension to  $  \bar \Om  \cap  B ( z_0, 4\rho ), $ 
   also denoted  $ \hat v, $  with  $ \hat v \equiv \hat  h $ on  $  \ar \Om 
\cap  B ( z_0, 4\rho ). $    If  
\[  
|\hat h ( z )  -  \hat h (w) | \leq    M'   | z - w |^{\hat \si} \quad \mbox{whenever}\quad    z, w \in \ar  \Om \cap B ( z_0,  4 \rho),  
\]  
for some $  \hat \si \in (0, 1], $  and $  1  \leq  M'  <  \infty, $ then there exists $  \hat \si_1 \in  (0, 1], $  depending only on $ \hat \si$ and $p $,  such that   
\begin{align}
\label{2.3a}   
 | \hat v ( z ) -  \hat v ( w) |  \leq  2 M' \rho^{ \hat \si}  +   \,    \,  (| z - w |/ \rho)^{ \hat \si_1}     \, \,    \max_{  \Om \cap  \ar B ( z_0, 2 \rho)} \, | \hat v|
\end{align}
whenever $z, w   \in \Om \cap B ( z_0,  \rho   )$. 
 
If  $   \hat h \equiv 0  \mbox{ on } \ar  \Om \cap  B ( z_0, 4 \rho)$,   $\hat v   \geq 0 $ in  $  B ( z_0, 4 \rho )$,    $\hat c  \geq  1,  $  and  $ z_1  \in  \Om  \cap B ( z_0, 4 \rho  ) $ with $ \hat c \,  d ( z_1, \ar \Om ) \geq    \rho, $  then there exists  $  \ti c,   $  depending only on $ \hat c $ and  $p$,    such that  
\begin{align}
 \label{2.3} 
\hs{.85in} (+) \hs{.2in}   \max_{B (z_0, 2 \rho )} \hat v  \leq  \ti c \left( \rho^{ p - 2} \,\int_{B ( z_0, 3 \rho)} \, | \nabla \hat v |^{ p } \, dx dy \right)^{1/p}   \leq   (\ti c)^2  \, \hat  v ( z_1 ).
\end{align} 
Furthermore, using \eqref{2.3a}, it follows for $ z, w  \in  \bar \Om  \cap B ( z_0, 2 \rho )$ that
\[
(++)   \hs{.2in}       | \hat v ( z  ) - \hat  v ( w ) |   \leq  c\,  \hat v (z_1)   \left( \frac{ | z  -  w |}{\rho } \right )^{ \hat \si_1}  . 
\]
\end{lemma} 
 \begin{proof}  
 For  the   proof of  \eqref{2.3a}   see  Theorem 6.44 in  \cite{HKM}.        Here
\eqref{2.3} $(+)$ is sometimes referred to as Carleson's  inequality,     see   \cite{AS}.
\end{proof} 
\begin{lemma}   
\label{lem2.3}	
Let $p,  \hat v, z_0,\rho,$ be as in Lemma \ref{lem2.1}.
Then	 $ \hat  v $ has a  representative locally in  $ W^{1, p} (B(z_0, 4\rho))$,    with H\"{o}lder
continuous partial derivatives in $B(z_0,4\rho) $  (also denoted $\hat v $), and there exist $ \hat  \ga \in (0,1]$ and $c \geq 1 $,
depending only on $ p, $ such that if $ z,  w  \in B (  z_0,  \rho/2 ), $     then  
\begin{align}
	\label{2.4} 	
\begin{split}	
&	(\hat a) \hs{.1in} \, \, 
 c^{ - 1} \, | \nabla \hat v  ( z ) - \nabla \hat v ( w ) | \, \leq \,
 ( | z -  w  |/ \rho)^{\hat  \ga }\, {\ds \max_{B (  z_0 ,\rho )}} \, | \nabla \hat  v | \leq \, c \,  \rho^{ - 1} \, ( | z -  w |/ \rho )^{\hat  \ga  }\, {\ds \max_{B (z_0, 2 \rho )}}  | \hat v |.  \\
&   \mbox{Also $ \hat v $  has distributional  second partials with } \\
&   (\hat b)  \, \,  \, {\ds \int_{B(z_0, \rho) \cap  \{ \nabla \hat v \not = 0 \} }}  \,   |\nabla \hat v  |^{p-2} \, \left(  |\hat v_{x x}|^2 (z)  + 
  |\hat v_{yy}|^2 (z)     +  |\hat v_{x y}|^2 (z)  \right) dx dy 
\leq   c \, \rho^{-p} \, {\ds  \max_{B (z_0, 2 \rho )}  | \hat v |} .\\
&(\hat c) \hs{.1in} \mbox{ If  $  \nabla \hat v ( z_0 )  \not = 0, $     then   $ \hat v $  is infinitely differentiable  in  $ B (  z_0, s ) $  for some  $ s > 0. $}  
\end{split}
\end{align}
 \end{lemma}

 \begin{proof}
  For  a  proof of \eqref{2.4} $(\hat a), (\hat b), $   see  for example  \cite{T}.  Now \eqref{2.4}    $ (\hat c)$ 
 follows from  $  (\hat a),  ( \hat b ), $  and Schauder type estimates (see \cite{GT}). 
 \end{proof}

\begin{lemma} 
\label{lem2.4}  
Let  $ x_0 \in \re$,  $\rho > 0$, $1 < p < \infty$,  and suppose $ \hat u$ and $\hat v$ are non-negative  
$p$-harmonic functions  in   $  \rn{2}_+ \cap  B  ( x_0, 4 \rho ) $    with  continuous  boundary values
$  \hat  v \equiv \hat u \equiv  0 $ on $ \re  \cap  B ( x_0, 4 \rho ).  $   There exists  $ c = c (p) $   such that  
\begin{align}
\label{2.5}   
\frac{\hat u (z)}{\hat v (z)}  \leq c   \frac{\hat u ( x_0 + \rho i)}{\hat v ( x_0 +  \rho i)}  
\quad \mbox{whenever}\quad   z \in  \rn{2}_+ \cap  B (x_0, 2\rho ). 
\end{align}
  Also $ \hat v $  has a $ p$-harmonic  extension to   
$ B ( x_0, 4 \rho) $  obtained  by  requiring    $ \hat v ( z )  = - \hat v  (\bar z )$ for $z  \in  B ( x_0, 4 \rho) \sem \rn{2}_+ .$    
\end{lemma}  

\begin{proof}  
Here \eqref{2.5} in   Lemma    \ref{lem2.4}   follows  from essentially 
barrier estimates for  non-divergence form  PDE.   See for example \cite{AKSZ}.  
The extension  process for $ \hat v $   is generally referred to  as  Schwarz reflection.  
\end{proof}        
    
Next  given  $\eta > 0$ and $x_0 \in \re$ let  
\[
S  (x_0,  \eta ) :=    \{ z = x+ i y :\, \,  | x - x_0 | < \eta/2, \, \,  0 < y < \infty \}.
\]     
For short we write $ S (\eta)  $  when  $ x_0 = 0. $   For fixed $p$, $ 1 < p < \infty$, let   $  R^{1, p } (S (\eta)) $  
denote the Riesz space of  equivalence classes of functions $ f $ 
on $ \rn{2}_+  $  with  $ f ( z + \eta  )  = f ( z )  $  when $ z \in \rn{2}_+  $  and  norm  
\begin{align}
\label{2.6} \| f \|_* =   \| f \|_{*, p}   =    \left( \int_{S ( \eta )}   | \nabla  f |^p  \, dx dy \right)^{1/p}   <  \infty. 
\end{align}
Also let $ R^{1,p}_0 ( S ( \eta )) $    denote functions in  $ R^{1,p} ( S ( \eta ) ) $ which can be approximated 
arbitrarily closely   in the norm of  $ R^{1,p} ( S ( \eta ))  $  by   functions in  this space which are  
infinitely differentiable and vanish in    an   open    neighbourhood  of $ \re. $      
  It is well known,   see \cite[section 1]{W},   that  given  $ f  \in  R^{1, p} (S (\eta)) $  there exists a  unique 
$p$-harmonic function  $ \ti v $   on  $ \rn{2}_+ $   with   $ \ti v ( z + \eta  ) = \ti v ( z )$ for $z \in  
\rn{2}_+  $  with $ \ti v   - f  \in R_0^{1,p} (S(\eta)).   $  In fact the usual minimization argument 
yields that $ \ti v $ has minimum norm among all functions $ h $  in  $ R^{1,p} ( S ( \eta ) ) $ with $ h - f 
\in R_0^{1,p} (S (\eta)).$    Uniqueness of $ \ti v  $ is  a  consequence of  the maximum  principle  in  Lemma  \ref{lem2.0}.   

Next we state 
\begin{lemma}   
\label{lem2.5} 
Given $ 1 < p < \infty, $ let  $ \hat v $  be  $p$-harmonic in  $ \rn{2}_+ $   and  $ \hat  v \in  R^{1,p} ( S ( \eta ) ). $  Then there exists $  c = c (p) $ and $ \xi  \in \re $  such that  
\begin{align}
\label{2.7}   
| \hat v  ( z ) -  \xi |  \leq c\,   \liminf_{t \to 0} \left(  \max_{ \re \times \{t\} }  \hat v  -   \min_{\re \times \{t\} }   \hat  v  \right)   \,  \exp \left(  -  \frac{ y}{c \eta}  \right) 
\end{align}
 whenever $ z = x + i y \in  \rn{2}_+ . $    
 \end{lemma}   
 
 \begin{proof} 
 Lemma \ref{lem2.5}  is proved in Lemma 1.3 of \cite{W} using $ \eta$  periodicity of  $\hat v$ and   facts about $p$-harmonic  functions   similar to 
  Lemmas  \ref{lem2.0} and \ref{lem2.1}.    
\end{proof}   

Finally, we state  (without proof)  an analogue  of  Lemma  \ref{lem2.5} for  $  B ( 0, 1) . $       

\begin{lemma}   
\label{lem2.6} 
Given $ 1 < p < \infty, $ let  $ \hat v $  be  $p$-harmonic in  $ B ( 0, 1), $     $ \hat  v \in  W^{1,p} ( B ( 0, 1)  ),  $  and  $  \hat v ( r e^{i\he} ) =  \hat v ( r e^{i (\he + \eta)}),  $ when  $ z = r e^{i\he} \in B ( 0, 1) $  and  $ 2\pi/\eta$ is a   positive integer.   Then there exists $  c = c (p)  \geq 1 $ such that  
\begin{align}
\label{2.8}   | \hat v  ( r e^{i\he} ) -  \hat v (0)  |  \leq  c\,  \liminf_{t \to 1}  (\max_{ B (0, t) } \hat v  -   \min_{B(0, t)} \hat  v  )   \, r^{\frac{ 1}{c \eta}} .  
\end{align}
\end{lemma}

\setcounter{equation}{0} 
 \setcounter{theorem}{0}
   \section{Proof of  Theorem \ref{thmA}} 
   \label{sec3}  
   In this section we prove  Theorem \ref{thmA} and as stated in the introduction we  
   give  two proofs  of Theorem \ref{thmA} when $2<p<\infty$.           
 An important role in each proof  is played by  homogeneous  $ p$-harmonic functions of  the form:   
\begin{align}   
    \label{3.1} 
      z= re^{i\he}  \,  \to \,  r^{\la}  \, \phi ( \he )\quad \mbox{for}\, \,   |\theta|  <  \alpha  \, \, \mbox{and} \, \, r>0, 
    \end{align}    
    satisfying  $\phi (0) = 1$, $\phi  ( \alpha )  = 0$,  $\ph ( \he )  = \ph ( -  \he )$,  $ \ph' < 0 $ on $ (0, \al]$, and  $ \ph \in C^\infty ( [- \al, \al ]) $     with   $\la   =  \la(\al) \in ( - \infty, \infty ).  $    Regarding \eqref{3.1},   Krol'   in \cite{K}    (see also \cite{A})  used \eqref{3.1} and  
    separation of variables  to show  for $ 1 < p < \infty, $                
\begin{align*}
0 =&  \frac{d}{d\he}  \left\{ [ \la^2 \phi^2 (\he) + (\ph')^{2} (\he) ]^{(p-2)/2}\, \ph' ( \he ) \right\}\,   \\
 &+  \la [ \la (p-1) + (2-p)  ]   [ \la^2 \phi^2 (\he) + (\ph')^{2} (\he)  ]^{(p-2)/2}  \ph ( \he )  .   
\end{align*}  
 Letting  $ \psi = \ph' / \ph $ in the above equation and proceeding  operationally  he obtained, the first order equation  
\begin{align} 
\label{3.2}     
\begin{split}
  0 = &       ( (p-1) \psi^2 + \la^2 )  \,  \psi'      \\
           & +  ( \la^2  + \psi^2 )   [(p-1)  \, \psi^2   +    \la^2   (p-1) + \la (2-p) ].
\end{split}
\end{align}
   Separating  variables  in \eqref{3.2}  one gets    
\begin{align}
 \label{3.2a}  
  \frac{  \la  d \psi}{ \la^2 + \psi^2 }     -  \frac{ (\la - 1 ) \, d \psi }{ \la^2  + \psi^2  + 
\la (2-p)/(p-1) } +  d \he = 0.
\end{align}
Integrating  \eqref{3.2a} and using   
 $\psi (0) = 0$  we obtain  for  $ 0 \leq  |\he| < \al $  
  that  
\begin{align}
 \label{3.3} 
 (\la/|\la|)   \arctan (\psi /\la) -  \frac{ \la - 1}{ \sqrt{ \la^2 + \la ( 2 - p)/(p-1)}}   \arctan \left( \frac{ \psi}{\sqrt{ \la^2 + \la ( 2 - p)/(p-1)}} \right)   = -  \he  
  \end{align}
   Letting    $ \he  \to   \al $  from the left  and using   $ \psi ( \pm \al ) = - \infty$  we get   
\begin{equation}  
\label{3.4} 
\pm 1  -     \frac{ \la - 1}{\sqrt{ \la^2 + \la (2 - p)/(p-1) }} =\frac{2 \al}{\pi}    
\end{equation}  
where $ + 1 $ is taken  if  $ \la > 0 $ and  $ - 1 $ if  $ \la < 0. $  Using the   
quadratic  formula  it  is  easily   seen   that  for fixed  $  \al   \in  (0,  \pi ]   $   
each  equation  has exactly one  $  \la  $  satisfying it  and  $  \la  > 0 $  if  the  +  sign is  taken while 
$ \la < 0 $ if the   -   sign is  taken   in   \eqref{3.4}.  
Using  these  values of  $  \la $  it follows that the operational 
argument can now be made  rigorous by reversing the  steps  leading to  \eqref{3.4}.  
 Then  \eqref{3.2}, $\psi (0) = 0,$   and  calculus  imply   that  
 $  \psi $  is decreasing  and negative  on  $ (0,  \al)  $.   
 Integrating   $ \psi $  over   $ [ 0, \he )$,  $\he < \al$, and  exponentiating   
 it follows that  $  \ph >  0 $  is decreasing  on  $ ( 0, \al  )$ with $ \ph (\al) =  0. $ 
Symmetry and  smoothness  properties of  $  \ph $  listed above 
can  be  proved using  ODE  theory  or  Lemma  \eqref{2.4} $(\hat c)$ and Schwarz reflection.         

To avoid confusion  later on   let   $ -  \hat \la $  denote the value of  $ \la $  in   \eqref{3.4}  
with    $-1$  taken, $\al = \pi/2,$  and   let $ \hat \ph $  
correspond to $ - \hat \la $ as in  \eqref{3.1}  for given $ p$,  $1 < p < \infty.$    
After some computation  we  obtain from  \eqref{3.4} that    
 \begin{align} 
 \label{3.5}       
 \hat \la  = \hat \la(p) =  (1/3) \left( - p  + 3  + 2  \sqrt{ p^2 - 3p + 3 } \,  \, \,  \, \right)  / (p - 1 ) .  
 \end{align}
 
\subsection{Proof of  Theorem \ref{thm1.2}} 
In this section we provide a hands on proof of Theorem \ref{thm1.2} when $2<p<\infty$. To this end, 
given  $  0 < t < 10^{-10},  $ let     $ a(\cdot)$  be  a   $ C^\infty $  smooth function on  $ \re $
 with  compact support in  $ ( - t, t )$, $0 \leq a  \leq 1$ with   $a  \equiv  1 $  on  $ ( - t/2, t/2 ), $  
 and   $ | \nabla a | 
\leq   10^5 /t .$   Let     $ f ( z )   = a(x) a(y) $  when $ z = x + iy \in \rn{2} $   and  
 for fixed $ p$, $1  <  p \neq 2< \infty$   let  $ \hat u $  be  the  unique   $p$-harmonic 
 function on $  \rn{2}_+ $   with  $ 0 \leq \hat u \leq 1$ satisfying 
\begin{align}
 \label{3.6}    
  \int_{ \rn{2}_+ }   |\nabla \hat u |^p   dx dy   \leq  \int_{ \rn{2}_+ }   |\nabla  f |^p   dx dy   \leq   c \, t^{ 2 - p},
 \end{align} 
and $\hat u  - f  \in   W_0^{1, p} (\rn{2}_+ \cap  B ( 0, \rho ) ) $  whenever  $  0 < \rho < \infty.  $ 
Existence and uniqueness of  $ \hat u $  
  follows with  slight modification  from  the usual calculus of  variations argument  for bounded domains (see \cite{E}).   
       We  assert    that  
    there exists $  \be_* \in  ( 0, 1]$    such that if    $  z, w \in B (0, \rho) \cap \mathbb{\bar R}_+^2 , $ then  
\begin{align}
 \label{3.7}  
  |  \hat u ( z )  -  \hat u (w) |  \leq   c   \left( \frac{|z-w| }{ \rho  } \right)^{\be_*} 
 \quad  \mbox{and} \quad    |\hat u(z)|  \leq  c   \left( \frac{t}{|z|} \right)^{\be_*}\,  \, \mbox{for} \, \,  z \in \mathbb{\bar R}_+^2 . 
  \end{align}
The left hand inequality in  \eqref{3.7}  follows  from  Lemma \ref{lem2.2}.   To  prove the  right hand 
inequality in   \eqref{3.7} observe from  the  boundary maximum principle in Lemma \ref{lem2.0} 
 and $  0 \leq \hat u  \leq 1, $  that 
$ {\ds \max_{\ar B(0, r)} \hat u }$  is  decreasing for     $ r  \in ( t,  \infty). $  Using this  fact 
and   Harnack's inequality in  Lemma \ref{lem2.1} $(c)$   applied to $ {\ds  \max_{\ar B(0, r)} \hat u   - \hat u} $, 
and  \eqref{2.3} $(++)$  we  deduce the existence of  $  \he \in (0, 1) $ with
\[  
\max_{\ar B(0, 2r)} \hat u  \leq   \he   \max_{\ar B(0, r)} \hat u. 
\] 
 Iterating this inequality we  get the right hand inequality in  \eqref{3.7}. 

 Next we  claim that    
\begin{align}
 \label{3.9} 
 \hat  u (i )   \approx    t^{\hat \la}
  \end{align}
where   $ \hat \la $ is as in   \eqref{3.5}.  To prove   \eqref{3.9},   let  $  z  = r e^{i\he}$ for $r > 0$ and $0 \leq \he \leq \pi$, and put 
\begin{align}
 \label{3.10} 
 v ( z ) =  v(re^{i\he}) = (t/r)^{\hat \la} \, \hat  \ph ( \he - \pi/2 )
 \end{align}
 where $\hat \la$ and $\hat \ph$ as defined before \eqref{3.5}. Then $ v $ is $p$-harmonic in  $ \rn{2}_+ $ with $ v \equiv  0 $  on $ \re\sem \{0\}$   and $ v ( i t) = 1. $  
Also from  Harnack's inequality and   
\eqref{2.3}    of   Lemma  \ref{lem2.2} with $ \hat v  = 1 - \hat u$,  $\hat u,  $  we find that 
$ \hat u (i t) \approx 1. $   In  view of the boundary values of  $  \hat u$, $v $  and  $ \hat u (it) \approx  v (it) = 1,$ 
as well as  Harnack's inequality in \eqref{2.2} $(c)$,   we  see that  Lemma  \ref{lem2.4}  can be applied to get    
\begin{align} 
\label{3.11}   
 \hat u/v   \approx 1  
 \end{align}
     in  $   \rn{2}_+  \cap [ B ( 0, 4t ) \sem B ( 0, 2t)]. $    
From      \eqref{3.7} for $  \hat u,  v, $   and  $ \hat \la > 0 $     we find first that  
$ \hat u(z), v(z) \to 0 $  as  $ z  \to  \infty $ in      $  \rn{2}_+  $    and thereupon from 
Lemma \ref{lem2.0}  that    \eqref{3.11}  holds in  $ \rn{2}_+ \sem \bar B (0,2t). $    
Since $ v ( i)  = t^{\hat \la} $   we conclude  from \eqref{3.11} that  claim  \eqref{3.9} is true.

Finally  observe from  \eqref{3.5} that     for $ 1 < p < \infty$             
\begin{align} 
\label{3.12}  
\begin{split}
(3/2)  (p-1)^2 (p^2 - 3p + 3)^{1/2} \,   d \hat \la/ dp &=  (p-1)( p - 3/2) - (p^2 - 3p + 3) - \sqrt{ p^2 - 3p + 3 }  \\
&= p/2 - 3/2   - \sqrt{ p^2 - 3p + 3 }  < 0.
\end{split}
\end{align}
    Indeed,  the inequality in the second line in \eqref{3.12} is clearly true if  $ p \leq 3 $ and  for $ p > 3 $ is true because  
\[ 
( p - 3 )^2 <  4  ( p^2 - 3p + 3 )  \mbox{ or   }  0  < 3 ( p^2 - 2p + 1) = 3 (p-1)^2. 
\]  
  Since  $ \hat \la ( 2 ) = 1 $  we  see that  
\begin{align}
 \label{3.13} 
  \hat  \la (p) > 1 \mbox{ for }  1 < p < 2  \quad \mbox {and}\quad \hat \la ( p ) < 1 \mbox{ for }  p > 2. 
 \end{align}

Let  $ \ti a $ denote the one periodic  extension of $ a|_{[-1/2, 1/2]}$   to  $ \re. $ 
 That is  $ \ti a  ( x + 1 ) = \ti a ( x )$ for $x \in \re$ and  $ \ti a = a $ on $ [-1/2, 1/2]. $ 
 Also  let   $  \Psi  $   be  the  $p$-harmonic function on  
$ \rn{2}_+ $  with 
\begin{align}
 \label{3.14} 
 \begin{split}
 & (a) \hs{.2in}    \Psi ( z + 1 )  = \Psi ( z ), \quad \mbox{whenever} \, \, z  \in \rn{2}_+, \\
&(b)  \hs{.2in}  \Psi  -  \ti a(x)  a(y)  \in  R^{1,p}_0 ( S(1))   \quad \mbox{and} \quad 0 \leq \Psi \leq 1  \mbox{ in } \rn{2}_+ ,   \\
&(c)  \hs{.2in}   { \ds  \int_{S(1)}  | \nabla \Psi|^p  dx dy   \leq  c \, t^{2-p}  <  \infty,   } \\
&(d)  \hs{.2in}    { \ds  \lim_{ y \to \infty} \Psi ( x + i y )  = \xi  \quad \mbox{whenever}\, \, x \in \re. }  
\end{split}
\end{align}
Existence of  $  \Psi $ satisfying $ (a)-(d)$ of \eqref{3.14}  follows  from the discussion after  \eqref{2.6}, 
and \eqref{2.7} of   Lemma \ref{lem2.5}  (see also \eqref{3.6} for $(c)$).       
Comparing  boundary values  of  
 $ \hat u$ and $\Psi$  we see  that  $  \hat u \leq  \Psi $  on   $ \re. $  Using  this fact and Lemma  \ref{lem2.0}  we find in  view of    \eqref{3.7} that 
\begin{align}
 \label{3.15} 
  \hat u    \leq \Psi \quad   \mbox{in}\quad    \rn{2}_+ .  
 \end{align}
 From   \eqref{3.15}, \eqref{3.9},   and  Harnack's  inequality for  $  \hat u,  $   we have   
\begin{align}  
\label{3.16} 
\int_{0}^{1}  \Psi (x +  i  ) \, dx  =  \int_{-1/2}^{1/2}  \Psi (x + i  ) dx  \geq \int_{-1/2}^{1/2} \hat  u (x + i ) \,  dx \approx t^{\hat \la}. 
\end{align}
   Also from \eqref{3.14} and  \eqref{2.3a}   we obtain     
\begin{align}
 \label{3.17} 
 \int_{0}^{1}  \Psi (x + s i  ) \, dx    =   \int_{-1/2}^{1/2}  \Psi (x + s i  ) \, dx  \leq    c \,  t 
 \end{align} 
for some small $ s > 0.$  Thus
\begin{align} 
\label{3.18}   
\int_{0}^{1}  \Psi (x + s i  ) \, dx   \leq  c  t^{ 1 - \hat \la}  \int_{0}^{1}  \Psi (x + i ) dx  
\end{align}  
where $ c $  depends only on  $ p . $   Recall from  \eqref{3.13}  that  $ \hat \la  < 1 $  if  $ p  > 2 . $   
 So  from  \eqref{3.18} and  \eqref{3.14} $(d)$   we see for $ t > 0  $  sufficiently  small  that  

\begin{lemma} 
\label{lem3.1}   
Theorem \ref{thm1.2}   is  valid for  one of the four functions       $ \Ph (z )  = \pm ( \Psi ( z  + i)  -  \xi )  $         
or, for $ s > 0 $  small enough,  $   \Ph (z )  =   \pm ( \Psi  ( z  + is ) - \xi )$ whenever $z \in \rn{2}_+$. 
\end{lemma} 
  
This  completes the hands on  proof  of  Theorem \ref{thm1.2}  when $ p > 2.  $

\begin{remark} 
\label{rmk1.2}  
The above   proof of  Theorem \ref{1.2}  fails when $1<p<2 $ as now $ \hat \la  > 1,  $   
so  $ t^{1 - \hat \la }  \to \infty  $ in \eqref{3.18}  as $ t \to  0. $     In short,  
our hands on example  could still be valid  for $ 1 < p  < 2, $    
but  in this case one needs  to make  a  better estimate than \eqref{3.18}.    
\end{remark}    

\subsection{Hands on proof of  Theorem  {\ref{thmA}}  when $2<p<\infty$}  

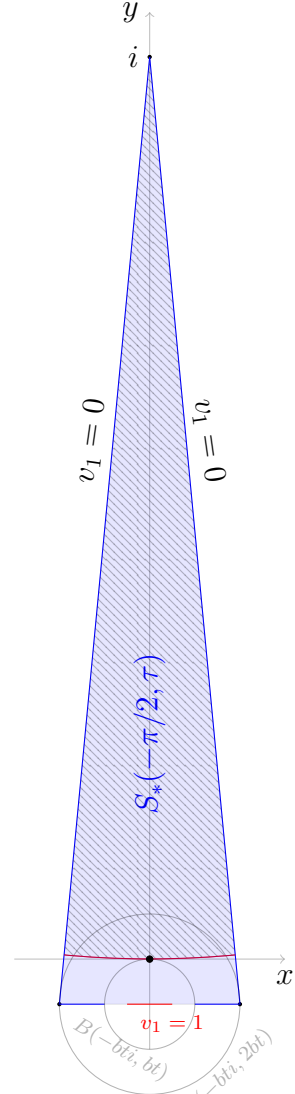
\begin{wrapfigure}{r}{0.4\textwidth}
\vspace{-20pt}
  \begin{center}
\begin{tikzpicture}[scale=1.2]
\coordinate (A) at (0,10);
\coordinate (b-) at (-1,-.5);
\coordinate (b+) at (1,-.5);

\draw[->, gray, opacity=.5] (-1.5,0)--(1.5,0);
\node[below] at (1.5,0) {$x$};
\draw[->, gray, opacity=.5] (0,-1)--(0,10.5);
\node[left] at (0,10.5) {$y$};
\draw[fill] (A) circle (.5pt) node[left] {$i$};
\draw[fill] (b-) circle (.5pt); 
\draw[fill] (b+) circle (.5pt); 

\draw[fill=blue, opacity=.1] (A)--(b+)--(b-)--cycle;
\node[blue, rotate=90] at (0,2.5) {$S_{*}(-\pi/2, \tau)$};
\draw[red] (-.25, -.5)--(.25,-.5) node[below] {\tiny $v_1=1$};

\draw[blue] (A)--(b+)--(.5,-.5);
\draw[blue] (-.5,-.5)--(b-)--(A);

\node[rotate=-82] at (.65, 5.75) {$v_1=0$};
\node[rotate=82] at (-.65, 5.75) {$v_1=0$};
\draw[opacity=.3] (0,-.5) circle (.5cm);
\node[opacity=.3,below, rotate=-30] at (-.2,-.8) {\tiny $B(-bti, bt)$};
\draw[opacity=.3] (0,-.5) circle (1cm);
\node[opacity=.3,below, rotate=45] at (.7,-1.1) {\tiny $B(-bti, 2bt)$};

\begin{scope}
\clip (A)--(b+)--(b-)--cycle;
\draw[purple, pattern=north west lines, opacity=.5] (A) circle (10cm);
\draw[purple] (A) circle (10cm);
\end{scope}
\draw[fill] (0,0) circle (1pt);
\end{tikzpicture}
  \end{center}
  \vspace{-20pt}
  \label{figure1}
  \caption{Domain $T$ and $S_{*}(-\pi/2, \tau)$}
\end{wrapfigure}

 \begin{proof} 
To   provide  examples  
 in  $ B ( 0, 1 ), $    satisfying  Theorem  \ref{thmA},  we need  to make somewhat  
 better estimates than in Lemma  \ref{lem3.1}  since  $p$-harmonic  functions  
 are  not invariant under dilatation in  polar coordinates.    For this purpose let    
 $ 0 <   b < <  t   < <  10^{-10} $.  For  the moment  we allow both  $ b$ and $t $ 
  to vary subject to these requirements but shall later  fix  $ t = t_0 $ and then 
  essentially  choose  $ b_0  < < t_0 $  so that if  $ 0 < b  \leq  b_0,  $  then  
  Theorem \ref{thmA}  is true for  our  examples.  To begin the proof,   let $ T   $  
  be  the triangular region whose boundary consists of    the horizontal  line 
  segment from  $  - b  -  b t \, i $  to   
$ b  - b t i $  and the  line segments joining $  i $  to  $  \pm b  -  b t i$ (see Figure 1).   
 Let    $ v_1 $ be the $p$-harmonic function in   $ T $    with  $ v_1 (z)  -   f ( z/b + t i)   \in W^{1,p}_0 ( T ) $ 
 where $ f $ is defined above   \eqref{3.6}.  Then from Lemma \ref{lem2.2} and  
 translation, dilation invariance of $p$-harmonic functions,  we see that  $ v_1$   
 has continuous boundary values with $  v_1  \equiv  1$ on  the open line 
 segment from   $ - bt/2  - b t i $ to   $  bt/2  - b t i, $ and $ v_1  \equiv 0 $ 
 on  $ \ar T \sem  \bar B ( - b t i,  bt ). $   
From the  definition   of  $ \hat u $ above  \eqref{3.6}   we find that   
\begin{flalign}
 \label{3.19}    
 \qquad v_1 (z )   \leq  \hat u ( z/b   + t i ) &&
 \end{flalign}
 in the  $ W^{1,p} $ Sobolev sense, when   $ z = x + iy  \in \ar T. $  
 Thus by  Lemma  \ref{lem2.0}  this inequality  holds in $ T. $   
 Also from \eqref{3.9},   \eqref{3.19},    and Harnack's inequality   we get            
\begin{flalign}
  \label{3.20}    
 \qquad v_1 ( b i )  \leq  c \, \hat u (  i ) \approx  t^{ \hat \la}. &&
  \end{flalign}    
  On the  other hand  since  both functions in \eqref{3.19}  have the same 
  boundary values  on  $  \ar T  \cap  \{ z = x -  b t i  :  - b   \leq  x \leq b \}  $ 
  it  follows   from \eqref{3.20},  \eqref{lem2.2},  Lemma \ref{lem2.4}, and 
   Harnack's inequality  that   
\begin{align} 
\label{3.21}   
\hat  u ( z/b  + t i)   \leq    c_+   (v_1( z )   +  t^{\hat \la}) \quad  \mbox{for}\, \, z  \in T \cap  \bar B (- i b t, 2 b ). 
\end{align}   
Also  from \eqref{3.11}  and the definition of $ v $  we have   
\begin{align} 
\label{3.22}   
\hat  u (i/\breve{c} )   \geq    2 c_+    t^{\hat \la}  
\end{align}
provided  $ t $ is small enough,  say  $ t \leq t_1,$ and  $\breve c $ is large enough   
where $ \breve c, t_1,  $  depend on  $ c_+ $  so only on $ p  > 2. $ 
   Using  \eqref{3.22}  in \eqref{3.21}  with  $  z   =  \frac{ -  \breve{c} \, t + 1 }{\breve{c}  }\, b  i $    we  obtain first that 
   $  v_1 (   \frac{ -  \breve{c}\,  t+ 1 }{\breve{c}  }\,  b  i )  \geq    t^{\hat \la}, $  and  
   second from Harnack's inequality for $ v_1 $ that   
\begin{align} 
\label{3.23}  v_1 ( b i )  \approx  t^{\hat \la}.  
\end{align}
Next if  $ \he_0 \in \re$ and $\eta  > 0$, we   let 
\[   
   S_* ( \he_0, \eta )  :=  \{ z: z = i + \rho  e^{i\he} :\, \,  0 \leq \rho < 1,\, \,   |\he - \he_0| < \pi  \eta  \}. 
\]    
   From high school geometry we  see that  if   $ \pi \tau   =  \arctan( \frac{ b}{1 +bt}), $ 
    then the rays  $ \he = -\pi/2 \pm \pi \tau $  drawn from $ i $ to  $ \pm  b  -  bt i $  
    make an angle $ \pi  \tau $ with the  $ y $ axis  and consequently   (see Figure 1)
\[ \bar T    \cap    \ar B (  i, 1 )  =  \ar S_* ( - \pi/2, \tau )  \cap  \ar B ( i, 1 ).  \] 
Given   $ N $   a  large  positive integer  choose  $ b $  so that $  \tau = N^{-1}  \approx  b. $    
We  claim  that  
\begin{align} 
\label{3.24}  
\int_{ \bar T \cap \ar B ( i, 1 ) }  v_1 ( z )  |dz|    \leq  c \, b \, t. 
\end{align}
 To prove \eqref{3.24}  we parametrize    $ \bar T \cap  \ar B ( i, 1)   $  by $  z (x)    =  x + i y (x)$ for $  
 - s   \leq   x  \leq  s$ where   $ s \approx b$ (so $y = 1  -  \sqrt{1-x^2}$).  Then from \eqref{3.19},   \eqref{3.11}, $ b < < t, $   
 and the fact that in \eqref{3.10},   $  \hat \ph (  \theta  - \pi/2) \leq  c \min (\theta  ,  \pi - \theta ) $  
 for $\theta \in [0, \pi], $      we see  as in   the proof of  \eqref{3.23}  that if $ 2bt  \leq  |x|  \leq  b$ then $|dz| \approx dx$ and 
\begin{align} 
\label{3.25}   
v_1 ( z (x) )   \leq c v ( z (x)/b + t i  )  \leq   c^2  ( b t)^{\hat \la} \,  |x|^{- \hat \la }  \, \left( \frac{|x|^2  +  b t}{  |x| } \right)  \approx   ( b t)^{1 + \hat \la}  |x|^{- \hat \la -  1  }.
\end{align}   
Thus   
  \[   
  \int_{ \bar T \cap \ar B ( i, 1 ) }  v_1 ( z )  |dz|  \leq c b t + c   \int_{b t}^b  ( b t)^{1 + \hat\la}   x^{- \hat \la   -  1 } dx 
 \leq   c^2   b t ,  \,  
 \]   
 so    \eqref{3.24} is true.    Let   $   \breve h (z)  = v_1 ( z) $ when $  z  \in \bar T  \cap \ar B ( i, 1)  $  
 and extend  $ \breve h $    to   $ \ar B ( i, 1) $  by  requiring that  
$ \breve h ( i + e^{ i \he} )  = \breve h ( i  +  e^{ i (\he + 2 \pi/N)} )$ for $\he  \in \re .$    
Let  $ \breve v $  be  the  $p$-harmonic function in  $ B ( i, 1) $  with  $ \breve v \equiv  \breve h  $ 
on $ \ar B ( i, 1) $ in the   $ W^{1, p} $  Sobolev sense.   From the usual  calculus of  variations argument we see that     
\begin{align} 
\label{3.26}
\begin{split}
&(a') \hs{.2in} 0 \leq  \breve v (i +  t e^{i\he} ) =  \breve v (i +  t  e^{i ( \he + 2\pi/N) } )   \leq 1 \quad \mbox{for}\, \, 0 \leq t  \leq  1, \, \,  \he  \in \re,    \\
&      (b')  \hs{.2in}   { \ds \int_{  S_*  ( - \pi/2, \tau)  }  |\nabla \breve v|^p  dx dy \leq  \int_{ T }   |\nabla v_1 |^p  dx dy \,}  \leq \, c \, (t/N)^{2 - p} .
\end{split}
\end{align}
  
 We   assert  that    
\begin{align}
 \label{3.27} 
 \begin{split}
 &(c')       \hs{.2in}  {\ds   \int_{ -  \pi}^{ \pi}  } \breve  v ( i  +  e^{i\he} ) d \he  \leq  c t,\quad      \mbox{ and }\quad   \breve v (i)   \geq   c^{-1} t^{\hat \la} ,  \\
 & (d') \hs{.2in}   | \breve v ( z ) - \breve v (w) | \leq  c \,(N/t)   \, \,    | z - w |  \quad \mbox{whenever}\, \,  z, w  \in  \ar B ( i, 1) .
 \end{split}
\end{align}
The left-hand inequality in  \eqref{3.27} $ (c') $  follows from  \eqref{3.24},   \eqref{2.3a}  of Lemma \ref{lem2.2},   
$ v_1 = \breve v $ on 
$ \bar T \cap \ar B ( i, 1), $  and  \eqref{3.26}  $ (a')$.  
To prove the  right-hand inequality in \eqref{3.27} $(c'),$ we note that   
$ \breve v \geq c^{-1} t^{\hat \la} $  on  $  \ar  B ( i, 1 -  1/N ), $  as we see from    
Harnack's inequality for $ 
\breve v, $  \eqref{3.23}, $ v_1 \leq  \breve v $ in $\bar T \cap \bar B (  i, 1  ) ,$ 
and  \eqref{3.26}  $(a')$.   This  inequality and the minimum principle for $p$-harmonic
 functions  give the   right-hand inequality in   \eqref{3.27} $(c').$    To prove  \eqref{3.27} $(d'),$   
 let  $  z \in  \bar T \cap \ar B ( i, 1 ), z_0 \in \ar T,  $  and suppose $ | z_0 - z| $  is   
 the distance from  $ z $ to  $ \ar T. $   If   $  | z - z_0 | \geq bt/4,  $  then  $ v_1 $  is  
 $p$-harmonic in  $  B ( z,  bt/4). $  Otherwise   
   from  Schwarz reflection we see that  $ v_1 $ has 
a $p$-harmonic  extension to  $ B ( z_0, b t/2)  $ (also denoted $ v_1). $  
 Thus in either case $ v_1 $ is $p$-harmonic in $ B ( z, bt/4) $  so  from \eqref{2.4} $(\hat a)$   of  Lemma \ref{lem2.3},  we   have 
\begin{align} 
\label{3.28}  
|  \nabla v_1 |   ( z)  \leq   c \, \,  (\max_{B ( z,  bt/4 )} |v_1| )  / bt  \leq  c/bt.  
\end{align}   
Now \eqref{3.28} and $ \breve  v = v_1 $  on $  \bar T \cap  \ar B  ( i, 1) $  give    \eqref{3.27} $ (d'). $  
    We now  choose  $ 0  < t_0  <  t_1 < 1  <    N_0, $ depending only on $ p > 2, $   so that if  $ N $ 
     is a positive integer with $ N  \geq N_0$, then  $  \eqref{3.26}, \eqref{3.27}, $  are valid  with $ t = t_0$ and also, 
\begin{align}
  \label{3.29} 0 \leq  \int_{ -  \pi}^{ \pi}    \breve v ( i + e^{i\he} ) d \he    \leq   \breve v (i)/2.  
  \end{align}
With  $ t_0 $ now fixed, put  
\[ 
 V ( z) =  \breve v(i)  -   \breve v ( z + i)  \quad \mbox{whenever} \, \, z \in B (0, 1). 
 \]     
 Then  from   \eqref{3.26}, \eqref{3.27}, \eqref{3.29}, we 
conclude that Theorem \ref{thmA}  is  valid for fixed $ p > 2$.   
\end{proof}     

\subsection{Finesse Proof of  Theorem A for $1 < p\neq 2 < \infty$} 
  In this section  we  give  a  proof  of  Theorem  \ref{thmA} valid for $ 1 < p < \infty , $  
  modelled on proofs of  Wolff \cite{W}  and  Varpanen \cite{V}, which however does not produce explicit  examples.   
   To this end,  we note  that   Wolff  (see also \cite[Section 3]{DB})   constructed  for fixed $ p$,  $1 < p\neq 2 < \infty$,   
   a  $p$-harmonic  function, $ F, $   of  the  form    
\begin{align}  
 \label{3.29a}     
 \begin{split}  
F (z ) =  F ( x + i y ) =  e^{ - \ga y }  f ( x)  \quad \mbox{for}\, \,  z \in  \rn{2}_+
 \end{split}
 \end{align} 
 where $\ga  > 0$ and $f$ satisfies
 \begin{align*}
 \begin{cases}
  f ( x  +  2\pi   )  = f  ( x  ) = f ( - x ), \\
  f ( \pi/2 - x )  = -  f ( \pi/2 + x ), \\
 f ( 0 ) = 1 \, \, \mbox{and} \, \,  f (  \pm \pi/2 ) = 0,\\
 f' (0) = 0\, \, \mbox{and}\, \, f'<0.
 \end{cases}
 \end{align*}
 Using $   F$ and $p-$harmonicity,   and separation of variables, it follows from  \eqref{3.29a}, as in   \eqref{3.2a},      that if  
$  \si (x)  = f' ( x )/ f ( x)$ whenever $x \in [0, \pi/2]$  then  
\begin{align} 
\label{3.30} 
-   \frac{d\si}{dx}  =   (p-1)  \frac{ ( \ga^2  +  \si^2)^2 }{  \ga^2  +  (p - 1) \si^2}, \quad  \si  ( 0 ) =  0\, \, \mbox{and}\, \,  \si (\pi/2) = - \infty 
\end{align} 
where the last equality means as a limit from the left.   Integrating \eqref{3.30}  we  get        
\begin{align}
\label{3.31}  
\frac{ p}{ 2 (p - 1) \ga }  \arctan (\si  (x) /\ga)   -  \frac{(p-2) \si  (x)  }{2 (p-1) (\si^2(x) +  \ga^2) }  =  -  x.  
\end{align}  
where we have used $ \si  (0) = 0. $   Letting  $ x \to  \pi/2 $ it follows from \eqref{3.31} and 
$ \si  ( \pi/2 )  =  - \infty $    that     
\begin{align} 
\label{3.32} 
 \ga = \frac{p}{2 ( p - 1)}.   
\end{align}
Next  we take   the  $ + $  sign  and   $ \al = \pi/(2N) $  in  \eqref{3.4}.     We obtain         
\begin{align}  
\label{3.33}     
1/N   =     1      -  ( 1 -  1/\la ) ( 1 - {\ts  \frac{ (p  - 2)}{\la (p-1)} } )^{-1/2}    \, \,  .
\end{align}    
Now  since  $ \la > 0$ and $N   >  1$,  we  see from   \eqref{3.33}  that  $  \la > 1. $    Using this fact and
 taking logarithmic derivatives of the   right-hand side of  \eqref{3.33} with respect to $ 1/\la, $ 
 we find that it  is    decreasing as  a  function of  $   1/\la . $  Thus  $  \la \to \infty  $  as  $ N  \to \infty.  $   
 Expanding   \eqref{3.33} in powers of $ 1/\la $  we  obtain   
\begin{align}
\label{3.34} 
\begin{split}
1/N  &=    1  -    ( 1 - 1/\la )  [   1   +    {\ts  \frac{  (p  - 2)}{2\la (p-1)} } +    O ( 1/\la^2 ) ]   \\
  &   =     \frac{ p}{2 (p-1) \la}  +  O ( 1/\la^2 )  \, \, \, \mbox{as} \, \, \, \la \to \infty.  
\end{split}
\end{align}  
    From  \eqref{3.34} we conclude that   
\begin{align}  
\label{3.35} 
 \frac{p}{2 (p-1)} N    =  \ga N  =    \la   +  O ( 1 )  \quad \mbox{as}\, \,  N  \to  \infty  
\end{align} 
 where $ \ga $ is as in  \eqref{3.32}.   Now suppose for the rest of the proof of  Theorem \ref{thmA} 
 that $ N \geq 10^{10} $ is   a positive  integer.    Let  
\[
    \la  =  \la ( \pi/(2N), p )\quad \mbox{and}\quad  \ph  =  \ph ( \cdot, \pi/(2N), p  )
\]
 be the value and  function in \eqref{3.1}  corresponding to  $  \al =  \pi/(2N). $   
 Then  $  \ph ( \pm  \pi/(2N) )   = 0 $  so from  Schwarz reflection  with $  \re $   replaced 
 by  $  \he =  (2k-1 ) \pi / (2N)$ for $k = 1, \dots,  N $  (see Lemma  
\ref{lem2.4}) it follows that  $ z = r e^{i\he} \to  r^{\la} \ph ( \he )  $ extends  to a  
$p$-harmonic function in   $  \rn{2}   \sem \{0\}, $  which is $ 2 \pi/N $  periodic in the $ \he $  variable. 
 Moreover since $ \la  > 1  $  in   \eqref{3.33},  we  see that  if  $ G ( z) =  G_N (z)  $ denotes this 
 extension and we define $  G ( 0) = 0, $  then   $ G $   is   $p$-harmonic in $ \rn{2}.  $   
 Let  $ g ( x  )  = g_N ( x )  = \ph ( x/N )$ for $x \in \mathbb{R}$   where we now  regard $ \ph = \ph ( \cdot, N ), $  
 as defined on $ \re. $   Then 
\begin{align} 
 \label{3.36} 
\begin{split} 
& (\al)  \hs{.1in}   g = g_N ( \cdot )  \mbox{ is  $ 2 \pi $  periodic on $ \re, $      $  g ( x )  = g ( - x )$,  $ g ( \pi/2 + x ) =  - g ( \pi/2 - x ),  $}  \\    
&\mbox{ \hs{.23in} for $  x \in \re,$  and  $ g' \leq  0  $  on  $ (0, \pi/2],   g (  \pm \pi/2) = 0, $}\\ 
& (\be)  \hs{.1in}    \mbox{ $    {\ds \max_{\re} }| g | = 1 = g(0) $ and  $ c^{-1 } \leq | g' ( x ) | + |g(x)| \leq  c , x \in \re,$  where $ c = c(p).  $ }  
\end{split}
\end{align}
Here   \eqref{3.36}  $ ( \al ) $  and the left hand inequality in  \eqref{3.36} $ ( \be), $ follow from the properties of $  \ph $ 
 listed after \eqref{3.1}   and discussed after  \eqref{3.4}.      To get the estimate from below in the 
 right hand inequality  of   \eqref{3.36} $ (\be) $ observe from Harnack's inequality and  \eqref{3.36} $(\al)$  
  that we only need prove this inequality for $ x $  near $ \pi/2$. Now comparing  $ G $ to a linear function
   vanishing on the rays $ \he = \pm \pi/2, $  using  Lemma \ref{lem2.4} with $ \hat u = G,  \hat v $  a linear 
   function vanishing  on  the ray $ \he = \pi/2, $ and taking limits as $ z   \to e^{i \pi/2}, $  we deduce  
   $ c^{-1} \leq g' ( \pi/2)  \leq c. $    The rest of  \eqref{3.36}   $(\be)$   follows  from \eqref{3.35} and   Lemma \ref{2.3}. We prove

\begin{lemma} 
\label{lem3.3}  
For fixed $ p$,  $1 < p\neq 2 < \infty$,  let $ f $   be as in \eqref{3.29a} and $ g = g_N  $   as in  \eqref{3.36}.  Then 
   $ g_N^{(k)}  ( x)   \to  f^{(k)} (x) $ as $ N  \to \infty, $   uniformly  on $  \re $ for  $ k = 0, 1, 2, \dots$.  
   \end{lemma}   
   
 \begin{proof}[Proof of Lemma \ref{lem3.3}]
 Given $ z = x +iy$,  $N $ a large positive integer,  and  $ G $ as defined  below   \eqref{3.35} 
 let  $ \ti u ( z )  = G  ( 1  +  i  z/N)$ for  $z  \in \rn{2}. $ From the definition of  $ \ph $  we see that if $ |z| < N, $   
\begin{align}  
\label{3.37}   
\ti u (z)  = \ti u_N ( x + i y ) =  \left[ (1- y/N)^2  + (x/N)^2 \right]^{\la/2} \,    
\ph \left( \arctan \left[  \frac{ x/N }{ 1 - y/N}  \right]  \right) . 
\end{align}  
Let 
\begin{align*}
 H ( z ) &:= H_N ( x + iy  ) =  \left[ (1- y/N)^2  + (x/N)^2 \right]^{\la/2},\\  
 K ( z ) &:= K_N  ( x + i y )  =     
\ph \left( \arctan \left[  \frac{ x/N }{ 1 - y/N}  \right]  \right)
\end{align*}
so   $  \ti u ( z )  = H ( z ) \, K ( z )$ when $ |z| < N. $    
Fix  $  R  > 100. $  Then from L' Hospital's rule, \eqref{3.37},  and   \eqref{3.35},  \eqref{3.36},     we find  uniformly   for  $ z \in B ( 0, 2R),$ that    
\begin{align}  
\label{3.38}  
 \lim_{N\to \infty}   H ( z  )  \,     =  e^{- \ga y}, \quad  \lim_{N\to \infty}   H_x  ( z  )  = 0,  \quad \mbox{and} \quad  \lim_{N \to \infty}  H_y ( z ) 
= - \ga  \, e^{-\ga y }.
\end{align} 
From \eqref{3.35},   \eqref{3.36}, and the same argument as  above  we see that   if  $ N'  $ is large enough  then 
  $  | \ti u_N |  $  is uniformly  bounded for  $ N \geq N',  $   so from   Lemmas  \ref{lem2.1} - \ref{lem2.3}, there exists  $ 1 < M < \infty $   with  
\begin{align}   
\label{3.39}  
\max_{B(0, 4R)} ( | \ti u_N | + |\nabla \ti u_N | )    \leq M <  \infty 
\end{align} 
for $ N \geq N'. $        From \eqref{3.39},   \eqref{2.4} $(\hat a) $,   and  Ascoli's theorem 
we see that a  subsequence say   $  (\ti u_{N_l}),  ( \nabla \ti u_{N_l} ), $  converges 
uniformly in   $ B (0, 2R ) $  to    $  u,  \nabla u,  $ and $ u $ is $p$-harmonic in  $ B ( 0, 2R).$     
Next we observe  that  $ (|H_N|)  $ is uniformly bounded below  in  $ B (0, 4R)  $  for $ N   \geq N'$ for $ N' $  large enough.   
Using   this  fact, \eqref{3.36},  and  \eqref{3.39} we see that        
\begin{align}  
\label{3.40} 
|K_x | (z)   =  N^{-1}     |\ph' |\left( \arctan \left[ \frac{ x/N }{ 1 - y/N}  \right]  \right)  
\frac{ ( 1 -  y/N) }{ ( 1 - y/N)^2  +  (x/N)^2 }   \leq M' <  \infty
\end{align}
 for $ N \geq N' ,  N   \in  \{ N_l \}. $      
Choosing $ y  = 0 $  in   \eqref{3.40}  and  using  \eqref{3.36},  properties of  $  \arctan$ function  we  deduce    
\begin{align}  
 \label{3.41}  
 |g'_N ( \hat x ) | =  N^{-1} | \ph' ( \hat x/N) |  \leq  2 M' \quad   \mbox{for}\, \,  \hat x  \in [-2R, 2R]. 
\end{align}   
From   \eqref{3.41}    and the  chain rule  it  follows easily that    
\begin{align}  
\label{3.42}   
\lim_{l \to  \infty}  (K_{N_l})_y  (z) = 0   \mbox{ uniformly  in } B ( 0, 2R) . 
\end{align}   
Thus  in  view of    \eqref{3.42},  \eqref{3.38},  we get     $   u  ( z ) = 
e^{- \ga y}  \nu (x)$ for $z \in  B  ( 0, 2 R),  $  so by uniqueness of     $ f $ in \eqref{3.29a} we have $ \nu \equiv f $ in 
$ B (0, 2R). $     Since  every subsequence of   $ ( \ti u_N ) $ converges uniformly to   $ F $  and $ R > 100, $   is  arbitrary  
we  conclude  Lemma  \ref{lem3.3}  when $ k  = 0. $

  Now from  \eqref{3.36} $(\be)$  and uniform convergence of  $ ( \nabla \ti u_N) $  to  $  \nabla F \not = 0 $ 
   on compact subsets of  $ \rn{2},$  we deduce  for $ N \geq N'$ that  $  \nabla  \ti u_N \neq  0 $  
    in $ B ( 0, R ). $     Then from   \eqref{2.4} $ (\hat c) $  we see  first  that  $  \ti u_N $  is infinitely 
    differentiable in  $ B (0, R ), $ for $ N  \geq N' $    and second   from  Schauder type arguments 
     using   \eqref{2.4} $ (\hat a), (\hat b), $  as in  \cite{GT},   that   
\begin{align}  
\label{3.43}    
D^{(l) }    \ti u   \to D^{(l)}  F  =   D^{(l)}   (  e^{-\ga y} f (x) ), \quad \mbox{for}\, \,  l = 0, 1,  \dots 
\end{align}
  uniformly on compact subsets of  $ \rn{2}$ where $ D^{(l)} $  denotes an arbitrary  $l$ th derivative 
  in either $ x $ or $y.$   To finish the proof of  Lemma \ref{lem3.3},   we  proceed by induction.    
  Suppose  by way of induction that  Lemma \ref{lem3.3}   is  valid for $ k  = l, $ a  non-negative integer.     
  Using the product formula for derivatives  and  \eqref{3.37} we find   that  taking   $ m $  partial derivatives in $x$ 
   on  $  H$  gives an expression that is  $ O ( N^{-m/2} ) $ when $ m $ is even  and  $ O ( N^{ - (m+1)/2 } )   $  
   when $ m $  is odd,   for  $  z \in  B ( 0, R )$     as  $  N  \to \infty. $   Also $ n \leq  l  $ derivatives on   $ K $   
    produces an  expression that is   
	$ O ( 1 ) $  in $ B (0, R) $  as  $ N   \to  \infty,   $  thanks to global  $ p$-harmonicity of  $ F. $  
	 Moreover in this $ O ( 1)  $  term  the  only way to  get  a  non-zero term in the limit as $ N \to  \infty $ 
	 is to put all derivatives on $ \ph,  $ which then gives from the induction 
	 hypothesis a   term  converging  to  $  f^{(n)} (x), $   as  $  N  \to   \infty.$         
	 From these  observations  and the product formula for derivatives we conclude   that    
\begin{align}   
\label{3.44}     
  \lim_{ N \to \infty}    \frac{ \ar^{l + 1} \ti u (z) }{  \ar x^{l +1}  }  =  \lim_{N \to \infty}  \left[ (1+y/N)^2  + (x/N)^2 \right]^{\la /2} \, g^{(l + 1 )}  (x)  =   e^{ - \ga y } f^{(l + 1) }(x). 
  \end{align}  
  From \eqref{3.44}, L' Hospital's rule, and induction    we see that Lemma \ref{lem3.3} is true.  
\end{proof}
    
In  order to use  Lemma  \ref{lem3.3}   we  briefly outline  Wolff's  proof  of  Theorem  \ref{thmA} for $  p > 2$ 
and also  the extension to $ 1 < p < 2 $ of this theorem  in  \cite{L1},    
tailored  to    $ 2\pi $  periodic rather than one periodic  $p$-harmonic functions on $  \rn{2}_+. $   
Let  $ F, f, \ga $ be as in  \eqref{3.29a}, $ 2 < p < \infty$,   and   for $ z  \in \rn{2}$ set   
\begin{align}  
\label{3.45}   
\begin{split}
A  (z) &= A ( x + i  y )\\ 
& =( (f')^2  +  \ga^2 f^2 )^{(p-4)/2}  e^{- \ga (p-2)y}   \left(   \bea{ll}       \ga^2  f^2 +  (p-1) (f')^2   &  - (p-2) \ga   f'  f \\
 - (p-2) \ga  f'  f  & (p-1)   \ga^2 f^2 +   (f')^2    \ea \right) (x).
 \end{split}
 \end{align}     
Note that $ A $ is  $ 2 \pi $ periodic in the $ x $   variable.  Moreover, if $  A (z) = ( a_{ij} (z) )$ for  
$z = x + i y \in \rn{2}_+ $ and 
$  \xi  = \xi_1 + i  \xi_2 \in \rn{2}, $  then   
\begin{align}  
\label{3.46}  
c^{-1} |\xi|^2 e^{- \ga  (p-2) y }   \leq   \sum_{i, j = 1}^2   a_{ij}  \xi_i \xi_j    \leq   c |\xi|^2   e^{ - \ga ( p - 2) y} 
\end{align} 
whenever $  \xi \in \rn{2}.$  Here \eqref{3.46}  follows from \eqref{3.29a}, Harnack's inequality for $F$, 
as well as  the analogue of   \eqref{3.36} $ (\be)$  for $ f.$    For the rest of  this section we regard
   $  \nabla  \psi  $ in rectangular coordinates,  as  a  $ 2 \times 1 $  column matrix whose top 
    entry is  $ \psi_x. $  Also,   $ \nabla \cdot $ is  a  $ 1 \times 2 $  row matrix whose first or 
    leftmost entry is  $  \frac{\ar}{\ar x}. $  Finally  if  $ \xi $ is a $ 2 \times 1 $  column matrix 
    and  $ \xi^t $  is the transpose of $ \xi, $ then  $  \lan A^* \nabla \psi,   \xi   \ran   =  \xi^t \, A^*  \nabla \psi $  
    whenever $ A^* $  is  a  $ 2 \times 2$ matrix with real entries.   
    
\begin{lemma} 
\label{lem3.4}
 Given  $ p$, $2 < p < \infty, $  there exists  $ \ze_i  =  \ze_i ( \cdot, p )   \in  C^{\infty} (  \bar {\mathbb R}_+^2 )$ for  $i  = 1, 2,  $   with  $  \nabla \cdot ( A \nabla \ze_i ) = 0 $ in  $  \rn{2}_+ $   satisfying  
\begin{align} 
\label{3.47}  
\begin{split}
&(\bar a) \hs{.2in} \ze_i  ( z + 2 \pi)  = \ze_i ( z ), \, \,  z \in  \rn{2}_+, \quad \mbox{and}\quad {\ds  \max_{\re} | \ze_i |}  = 1. \\
&(\bar b) \hs{.2in}   { \ds \int_{S(2\pi) }   \lan A \nabla \ze_i, \nabla \ze_i  \ran dx dy      \approx  \int_{S(2\pi)}  e^{- \ga ( p - 2) y}\,  | \nabla \ze_i  |^2 dx dy}  < \infty \, . \\
& (\bar c) \hs{.14in}  \mbox{There exist  $  \de = \de(p)  \in (0, 1]$ and $  \mu_i  \in \re$  with}\, \,       {\ds  \lim_{y \to \infty }    \ze_i ( x + i y )  = \mu_i \, \, \mbox{for} \, \, x\in\re}  \\   
&\hs{.4in}  \mbox{and}\, \,     | \ze_i  ( z ) - \mu_i   |  =  | \ze_i ( x + i y) - \mu_i   |  \leq   2  \, e^{- \de y}  \, \, \mbox{for}  \,  \, y \geq 0.  \\
&    (\bar d)   \hs{.2in}   {\ds   \max_{ \rn{2}_+ } } \,  |\nabla \ze_i| \leq M < \infty  \quad \mbox{and} \quad   {\ds \int_{S(2\pi)}   | \nabla \ze_i |^q \, dx dy}  \leq   M_q   < \infty   \mbox{ for  }  q \in ( 0,  \infty).  \\ 
&    (\bar e)  \hs{.2in}   \mbox{There exist  $ y_0  \in (0, 1), $  $  c_+$, and $c_{++}  \geq 1, $    with } \, c_+^{-1}  \leq  {\ds \int_{- \pi}^{\pi} }   {\ts\frac{\ar \ze_1}{\ar y}} (x + i y) \, dx  \leq  c_+  \mbox{ and }  \\ 
& \hs{.4in} c^{-1}_{++}  \leq  {\ds \int_{-\pi}^{\pi} } \lan  A   \nabla  \ze_2,   e_1  \ran     (x + i  y )  \,  dx   \leq c_{++}  \mbox{ for }  0 \leq y \leq y_0,    \mbox{  where $ e_1 = 
\left( \bea{l} 1 \\ 0 \ea \right) $.}  
\end{split}
\end{align}
 \end{lemma}     
 
 \begin{proof}[Proof of Lemma \ref{lem3.4}]    
 The proof of  Lemma \ref{lem3.4} for $ \ze_1$  and essentially also for  
    $ (\bar a)-(\bar d)$ of  $ \ze_2, $  is given in section 3 of \cite{W}.   
     The proof  of  $(\bar e)$ in  Lemma  \ref{lem3.4} for  $ \ze_2  $   is in  \cite{L1}.    
 \end{proof} 

Next   for for  fixed $ p$, $ 2<p<\infty$ and $\la = \la ( N, p )$, let $    T^p ( S ( 2 \pi )) $  be equivalence classes of   
functions  $  h  $ on $ \rn{2}_+ $   with  $ h ( z + 2 \pi )  = h ( z)$ for $z \in \rn{2}_+$,   distributional partial derivatives $ \nabla h,  $  and norm, 
  \[ 
  \| h \|_{+,p}   =   \int_{S ( 2 \pi)}  e^{ -  ( \la - 1) (p-2) y/N} \, | \nabla h |^2  ( x + i y)  dx dy  <  \infty.  
  \]
Also let  $ T_0^p ( S ( 2 \pi ))  \subset  T^{p} ( S ( 2 \pi ) ) $ be functions in this space that can be  
approximated arbitrarily closely in the above norm by  $ C^\infty $  functions in $ T^p (S (2 \pi))$  
that vanish in an open neighbourhood of  $  \re. $  
For  $g = g_N  $   as in  \eqref{3.36}  and $ z  \in \rn{2}_+$ set

\begin{align}
 \label{3.48}
 \begin{split}   
 \breve A  (z) =& \breve A_N  ( x + i  y )  \\ =&
( (g')^2  +  (\la/N) ^2 g^2 )^{(p-4)/2}  e^{- (\la - 1 ) (p-2)y/N}   \quad        \\
&  \times\left(   \bea{ll}       (\la/N)^2  g^2 +  (p-1) (g')^2   & - (p-2) (\la/N)   g'  g \\ \\
 - (p-2) (\la/N)   g'  g  & (p-1)   (\la/N)^2 g^2 +   (g')^2    \ea \right) (x) .
\end{split}
\end{align}
From   \eqref{3.36}  we observe that  $ \breve A ( z + 2 \pi)  =  \breve A ( z )$ for  $z \in \rn{2}_+ $ and from \eqref{3.35},  
\eqref{3.36}, Lemma  \ref{lem3.3},   that  if $  \breve A (z)  =  ( \breve a_{ij} (z)), $ then    \eqref{3.46} holds with  $ a_{ij} $  
replaced by  $ \breve a_{ij} $   provided  $ N \geq N'  $ and $  N' $ is large enough.  
   Let   $  \breve \ze_i  = \breve \ze_i ( \cdot, N ) $  be  the weak solution to   $  \nabla \cdot  (\breve A   \nabla \breve \ze ) = 0  $ in $ \rn{2}_+ $ with  
  $ \breve \ze_i -  \ze_i  \in  T^p_0 ( S ( 2 \pi ) ). $   Existence and uniqueness  of  $ \breve \ze_1,  $     
for example,   follows from \eqref{3.46} for   $  \breve A $  and  a  slight modification of   the usual  
calculus of  variations minimization argument  often given  for  bounded domains.   To  indicate this modification,   let  
\[ 
I ( h) =   \int_{ S ( 2\pi) }  \sum_{i, j = 1}^2  \breve a_{ij} \, h_{x_i}   h_{x_j} dx dy   
\] 
where the functional $I(\cdot)$  is  evaluated at functions  in  
\[   
\mathcal{F} : =   \{h:\, \,  h  \in  T^p ( S( 2 \pi ) )  \mbox{ with } h -  \ze_1 \in  T_0^p ( S (2 \pi ) ) \}.   
\]   
For fixed   $ \rho  > > 2 \pi, $  one  can   choose  $   h_j  \in  \mathcal{F}$ for $j = 1, 2, \dots $  so that  
\begin{align} 
\label{3.45a} 
\begin{split}
\begin{cases}
\, \,  {\ds \lim_{j \to \infty}  I  ( h_j ) =  \inf \{ I ( h) : h \in  \mathcal{F} \}},    \\
\, \,   {\ds h_j |_{S ( 2 \pi) \cap B(0, \rho )}    \rightharpoonup \ti h  \quad  \mbox{weakly in}\, \,  L^2 ( S ( 2 \pi) \cap B(0, \rho ) ), }   \\  
\, \,  \mbox{Each component of  $ \nabla  h_j $  tends  weakly  to  a function in  $L^2 ( S ( 2 \pi)) .$ } 
\end{cases}
\end{split}
\end{align}
Integrating by  parts  and using the definition of  a  distributional derivative,  it follows from  
\eqref{3.45a} that   $  \nabla  \ti h $ exists in the  distributional sense and   $ \nabla h_j  \rightharpoonup  \nabla \ti h $  
weakly  on  $ S ( 2 \pi) \cap B(0, \rho ).  $   Moreover  using     
$ h_j - \ze_1  \in T_0^p ( S ( 2 \pi)) $   it  follows that   $ \ti h $  can be chosen independent of  $  \rho $.  
 Using  lower semicontinuity of the functional we conclude that   $  \ti h  \in \mathcal{F} $  
 and  $ I (\ti  h )  =  \min_{h \in \mathcal{F}} I ( h). $    
The  rest of the proof is unchanged from the usual one for  bounded domains.      

Using  \eqref{3.46} for   $ \breve A, $       elliptic  
regularity theory, and  Lemma  \ref{lem3.3},  we also  find for $ N  \geq N'$ and $i=1,2$ that
\begin{align} 
\label{3.49}  
\begin{split}           
&\breve \ze_i  =  \breve \ze_i ( \cdot, N )   \in  C^{\infty} (  \bar {\mathbb R}_+^2 )   \, \, \mbox{with} \, \,  \breve \ze_i ( z + 2 \pi ) =  \breve \ze_i ( z )\,\, \mbox{and}  \, \, {\ds  \max_{\rn{2}_+}}\,  | \breve \ze_i  | = 1 , \\ 
&\mbox{   and  $ \nabla \cdot ( \breve  A \nabla  \breve \ze_i ) = 0 $ in           $ \bar {\mathbb R}_+^2 $ .  }  
          \end{split}
          \end{align}
Furthermore,  arguing as in section 3  of  \cite{W}    we  deduce the  existence of   $  \breve  \de  \in (0, 1], $ 
depending only on $p, $  and   $  \breve \mu_i  \in [-1,1]$ for $i = 1, 2,  $   satisfying      
\begin{align} 
\label{3.50}  
\lim_{y \to \infty}   \breve \ze_i  ( x + i y)  =    \breve  \mu_i \, \, \mbox{for}\,   \, x  \in \re \quad  \mbox{and}\quad    
     |  \breve  \ze_i \, ( x + i y)  \, -  \, \breve \mu_i  |  \, \leq  e^{- \breve \de y}  \, \,  \mbox{for}  \,  \, y \geq 0.  
\end{align}  
From Lemma \ref{lem3.3} and \eqref{3.35}  we see that   $  D^{(l)}  A   \to   D^{(l)} \breve A$ for $l=0,1,  \dots, $ 
 uniformly on $  \bar {\mathbb R}_+^2 $ as $ N  \to \infty,  $  where 
$ D^{(l)}  A $ denotes an arbitrary $l$-th partial derivative of $ A.$  From this observation,   
\eqref{3.50},  \eqref{3.49},  elliptic regularity theory,  and  Ascoli's theorem it follows  that  
\begin{align}  
\label{3.51}  
\lim_{N \to \infty}  D^{(l)}  \breve  \ze_i ( \cdot, N )    \to     D^{(l)} \ze_i  \mbox{ uniformly in $ \bar {\mathbb R}_+^2 $ as $ N \to \infty, $   for 
$ l = 0, 1, \dots $ }  
\end{align}

In view of    \eqref{3.51} and  \eqref{3.47} $ (\bar d), (\bar e ), $ we see for $ N' $  large enough and $ N \geq N', $   that     
\begin{align}
\label{3.52}  
\begin{split}
&(\al)  \hs{.2in}     {\ds  \max_{\rn{2}_+}  | \nabla \breve \ze_i|  \leq \breve M < \infty \mbox { and }  \int_{S(2\pi)} }  | \nabla \breve  \ze_i |^q dx dy  \leq  \breve M_q   < \infty   \mbox{ for  }  q \in ( 0,  \infty).   \\ 
&(\be) \hs{.2in}   c_*^{-1}   \leq  {\ds  \int_{-\pi}^{\pi}  \,  {\ts\frac{\ar \breve \ze_1}{\ar y}} (x + i  y) \, dx  }  \leq c_*  \mbox{ for } 0 \leq y \leq y_0,   \\
&(\ga) \hs{.2in} c_{**}^{-1}  \leq   {\ds  \int_{-\pi}^{\pi}  \lan \breve A \nabla \breve \ze_2 ,  e_1 \ran  (x + i  y) \,  dx  } \leq  c_{**} \mbox{ for }  0 \leq y  \leq y_0  \ .
\end{split}
\end{align}  
 Constants  in  \eqref{3.52}  are independent  of $ N \geq N'$ and $ \breve M,  \breve M_q,  c_* , c_{**}, $ 
 depend only on  $ p,  $  as well as the corresponding constants for $ \ze_1, \ze_2, $   in   
 (3.47) $ (\bar d), (\bar e) $  of Lemma \ref{lem3.4}. 

  To continue the proof of Theorem \ref{thmA}  for $ 1 < p\neq 2 < \infty, $ given $ z  = r e^{i\he}  \in B ( 0, 1), $  $ N \geq N', $    
  we follow   \cite{V}   and let  $ w = w (z) =     N  \he  -  i  N \log r  \in  \rn{2}_+ . $ 
  Then  $ w  $  maps  $ \{  z = r e^{i\he} : 0 <  r  <   1, |\he| < \pi/N  \}   $   one-one  and  onto  $ S ( 2 \pi ). $  
     If $  z = x+ i y =  r e^{i\he},$   put    $   \ti A (z)  = \breve A (w (z) ),$ when $    z \in B (0, 1) \sem \{0\}, $ and $ \ti A (0) = 0.$  
We note  from \eqref{3.48} for    $ z  \in B ( 0, 1)$ that 
  \begin{align}
  \label{3.53} 
  \begin{split} 
   \ti  A  (z) &= \breve A_N  (  N \he - i N \log r ) \\  
 &= \tau ( re^{i\he} ) 
     \left(   \bea{ll}       
     \la^2  \ph^2 +  (p-1) (\ph')^2   & - (p-2) \la   \ph' \ph  \\ 
 - (p-2) \la   \ph'  \ph  & (p-1)   \la^2 \ph^2 +   (\ph')^2    
 \ea \right) 
 ( \he )  
\end{split}
\end{align}  
 where 
 \[   
 \tau ( r e^{i\he} )     \, =    \,   N^{(2-p)} r^{ (\la - 1 ) (p-2)} ( (\ph')^2  +  (\la) ^2 \ph^2 )^{(p-4)/2}  (\he) .   
 \] 
Here  $   \la =  \la(2\pi/N) $  and  $ \ph =  \ph (  \cdot,  N  )$  is the extension of  $  \ph = \ph ( \cdot, 2\pi/N) $
  in \eqref{3.1}  to  $   \re.$   Let  $ \ti \ze_i (z) = \breve \ze_i ( w (z)  )$ for  $z \in  B ( 0, 1 )  \sem \{0\} $ 
  and observe that $ \ti \ze_i $ is  $ 2 \pi/N$  periodic in the $  \he $  variable.  
  From  the chain rule,  \eqref{3.49} - \eqref{3.52} and   elliptic interior regularity estimates  
  for $ \breve \ze_i,  i = 1, 2, $ akin to  \eqref{2.4} $ ( \hat a) $ of   Lemma  \ref{2.3}      
  we  see for $ i = 1, 2, $   that 
\begin{align}
\label{3.54}    
\max_{\ar B ( 0, r )}  [ ( r/N)  \, |\nabla \ti \ze_i| +  | \ti \ze_i  - \breve \mu_i  | ]    \leq  \ti M  \,  r^{ N \de }  \quad \mbox{for}\, \,  0 < r \leq 1, 
\end{align} 
where $ \ti M $  is independent of  $ N $  for  $ N   \geq N' . $    
Put  $ \ti \ze_i (0) = \breve \mu_i. $ Then  \eqref{3.54},   \eqref{3.52} $(\al)$, and the chain rule imply that    
\begin{align} 
\label{3.55}    
 {\ds \int_{B (0, 1) } }  | \nabla \ti   \ze_i |^q r dr d \he   \leq    N^{q-1}  \ti M_q   < \infty  \quad  \mbox{for}\, \,  q \in ( 0,  \infty), 
 \end{align} 
 where  $ \ti M_q $ is independent of   $ \ti \ze_i $  for  $ N \geq N' . $   Also  from 
 \eqref{3.52} $(\be), (\ga),  \eqref{3.53}, $ we  deduce for $ N \geq N', $    that  
\begin{align}
\label{3.56} 
\begin{split} 
&(+) \hs{.32in} 
 (2 c_*)^{-1} N   \leq  {\ds  \int_{-\pi}^{\pi}  \,  {\ts\frac{\ar \ti \ze_1}{\ar r}} (r e^{i\he}) \, d \he }  \leq 2 c_*  N  
\quad \mbox{for}\, \,   1  - \frac{y_0}{2N}  \leq r   \leq 1,  \\ 
& (++)   \hs{.2in} (2 c_{**})^{-1} N \leq   {\ds  \int_{-\pi}^{\pi}   \lan  \ti A \nabla' \ti \ze_2, \, e_1\ran  (r e^{i\he} ) \,  d\he   } \leq  2 c_{**} N \quad \mbox{for}\, \, 1  - \frac{y_0}{2N}  \leq r   \leq 1,
\end{split}
\end{align}  
where 
\[  
\nabla'  \ti  \ze_2 ( r e^{i\he})  =   \left( \bea{l}\hs{.03in}  r^{-1}  \frac{\ar \ti \ze_2}{\ar \he} \\-  \frac{\ar \ti \ze_2}{\ar r}  \ea \right).
\]    
Next  we observe from  $  \nabla \cdot ( \breve A  \nabla \breve \ze_i ) = 0 $ for $ i = 1, 2, $  
and the  change of  variables formula   that     if $  \chi \in C_0^{\infty}  ( B ( 0, 1 ) \sem \{0\} )$  then    
\begin{align}  
\label{3.57} 
I   =  \int_{B(0, 1)}    \lan \ti  A  \nabla' \ti \ze_i,  \nabla' \chi \ran r dr d \he  = 0  
\end{align}
From   \eqref{3.54}  and the usual limiting arguments  we see  that   \eqref{3.57} still holds if $  \chi \in C_0^{\infty}  ( B (0, 1) ). $  
Finally, if  $  \bar v (z) =  \bar v (re^{i\he})    =      r^{\la}   \ph ( \he, N ),   $    and   
\begin{align} 
   \label{3.58}
  \begin{split}   
   \bar  A  (z) =      |\nabla \bar v |^{p-4}  \left(   \bea{ll}      (p-1)  \bar v_x^2  + \bar v_y^2     &  (p-2) \bar v_x  \bar v_y  \\
  (p-2)    \bar v_x  \bar v_y  & (p-1)   \bar v_y^2 +  \bar v_x^2    \ea \right) ( z)  
\end{split}
\end{align}
  when  $ z = r e^{i\he} \in B ( 0, 1)$. Then   \eqref{3.57} can be rewritten as   
\begin{align}  
\label{3.59}   
I   =    \int_{B ( 0, 1)}   \lan  \bar A  \nabla \ti \ze_i,  \nabla \chi \ran  dx  dy  = 0  
\end{align}   
so  $  \nabla \cdot ( \bar A \nabla \ti \ze_i  ) = 0 $ in  $ B ( 0, 1). $ Here \eqref{3.59} can be verified  
by using the chain rule to switch  \eqref{3.57} from polar to rectangular coordinates
 but also as in \cite{V}   by  noticing   that  if 
   $  a ( \cdot, \ep )  =  \bar v  +  \ep   \ti l$ for    $\ti l  \in  \{ \ti \ze_i ,   i = 1, 2 \},  $    
    then   
     \[  
     \frac{\ar}{ \ar \ep}  \left( \nabla \cdot ( |\nabla a|^{p-2} \nabla a ) \right)_{\ep = 0}  =  \nabla \cdot  ( \bar  A  \, \nabla \ti l  ) = 0.  
     \]   
     The left hand side of  this  equation  can be  evaluated independent  of  the coordinate system,   
     so  letting  $ \bar v_{\xi}$ and $\bar v_{\eta}$ denote directional derivatives of   $ \bar v $ at $ z,$  where 
    $ \xi  = i e^{i\he}$ and $\eta = - e^{i\he}$,   we obtain
    \[
    \bar v_{\xi} = r^{-1}  \bar v_{\he} \quad \mbox{and} \quad \bar v_{\eta} = - \bar v_r .
    \]
    Using this fact,  replacing $ \bar v_x$ and $\bar v_y$ in   \eqref{3.58} and \eqref{3.59}    
    by $ \bar v_{\xi}$ and $\bar v_{\eta}, $   and computing  $ \nabla \ti \ze_i$ and $\nabla \chi, $ in the  $ \xi$ and $\eta$ 
    coordinate system, we arrive  at  \eqref{3.57}. Moreover,   \eqref{3.56} $(++)$ can be rewritten as   
\begin{align} 
 \label{3.60}     
 \hs{.2in} (2c_{**})^{-1} \la^{p-2} N \leq   {\ds  \int_{-\pi}^{\pi}   \lan  \bar  A \, \nabla \ti \ze_2, \, e_{\he} \ran  ( r e^{i\he}  ) \,  d\he  } \leq  2 c_{**}  \la^{p-2} N   
 \end{align} 
 for $1 -  { \ts \frac{y_0}{2N}} \leq r \leq 1$ where  
 \[
  e_{\he}  =  \left(  \bea{l} - \sin \he \\ \hs{.05in} \cos \he \ea \right).  
\]

Armed with   \eqref{3.54}-\eqref{3.56} and   \eqref{3.59},   we can now essentially copy  
the proof of  Lemmas 3.16-3.19 in \cite{W} for $2<p<\infty$  and  the argument leading to    (12)-(13)  in \cite{L1} for $1<p<2$. 
Thus  the reader should have these papers  at hand.  Since  constants now depend on  $ N, $  
we briefly  indicate  the slight changes in lemmas and displays.  In the proof we let $ C \geq 1  $ be a constant, 
not necessarily the same at each occurrence,  which may depend on   other  quantities besides $p, $  
such as   $ c_*, c_{**}, $  but is independent of  $ N $  and $ \ep, $   for  $ N  \geq N'$,  $0 < \ep \leq  \ep'. $     
 Given $ p$, $1 < p\neq 2 < \infty, $ and $ \ep > 0$  small, for $i=1,2,$  let  $  k_i  =  k_i ( \cdot, N ) $  be the  $ p $-harmonic function in  $ B ( 0, 1 ) $ with  
$ k_i  = \bar v  + \ep \, \ti \ze_i$  on $ \ar B ( 0, 1)   $    in the  $ W^{1,p} $  Sobolev sense.     From Lemma  
\ref{lem2.2}   we see that  $ k_i  $  is H\"{o}lder continuous in $ \bar B (0, 1). $  Also from 
the boundary maximum principle for $ p $-harmonic functions we deduce for $z = r e^{i \he} \in \bar B ( 0, 1)$ that  
\[
k_j (r e^{i\he}) = k_j  ( r e^{ i (\he + 2 \pi/ N)}) \quad \mbox{for}\, \, j = 1, 2.
\]
We note that  $ f$, $v$, and $g $ in  Wolff's notation in \cite{W} corresponds to our  $ \bar v$,  $\ti \ze_i$, and $k_i$ respectively.  
 If   $ q \in W_0^{1, p} ( B (0, 1 ) )$ and $2<p<\infty$  then the  analogue of 
the display in Lemma 3.16 of  \cite{W}  in our notation relative to  $ B (0, 1) $  is      
\begin{align} 
\label{3.61}   
 \left| \int_{B(0, 1)}    \lan \nabla q, \nabla ( \bar v + \ep \ti \ze_i  ) \ran | \nabla ( \bar v + \ep \ti \ze_i)|^{p-2}   dx dy  \right|  \leq  C  \ep^{\si}\,  N^{ (p - 1)/p' } \, \| | \nabla q |\|_p  
 \end{align}
for $ N \geq N', $ where  $ p' = p/ (p-1)$,  $\si = \min ( 2, p - 1),  $ and  $ \| | \nabla  q | \|_p $ is the Lebesgue  $p$ norm of  $  | \nabla q | $ on $ B ( 0, 1). $  
 To get this estimate  we   use H\"{o}lder's inequality,  \eqref{3.54},  and our knowledge of  $  \bar v $  
  to estimate the  term in brackets in  display  (3.17) of  \cite{W}.

Lemma  3.18  of this paper  follows easily from Lemma 3.16  with 
 $ q =     \bar v  +  \ep  \, \ti \ze_i  -  k_i$ for  $i = 1, 2,  $  and now reads, 
\begin{align}   
\label{3.62}      
\| |  \nabla \bar v  + \ep \nabla \ti \ze_i | \|_p^p   \leq   \| | \nabla k_i | \|_p^p      +  C  \ep^{\si} \,  N^{(p-1)/p'} \,  \| | \nabla  k_i  -  \nabla \bar v - \ep  \nabla \ti \ze_i  | \|_p    
\end{align}    
where all norms are relative to $ B ( 0, 1). $   
 
The  new version of the  conclusion in Lemma 3.19 of  \cite{W}  is: {\em There exists $  \ep' \in ( 0, 1/2)$ and     $  C  \geq  1 $    such that},    
\begin{align}  
\label{3.63}  
\int_{ B (0, 1  ) \sem B ( 0,  1 - \frac{y_0}{2N} ) }  | \nabla ( \bar v   + \ep  \ti \ze_i  ) -  \nabla  k_i   | \, dx dy
\leq   C  \ep^{\ti \tau}  
\end{align}   
{\em for $ 0 < \ep \leq \ep'$   where  $  \ti \tau =  \si p'/2 > 1. $}   To get  this new conclusion  first  
replace  $  \ep^{\si}$ by    $  \ep^{ \si }  N^{(p-1)/p' } $  and $ S^\la $ by $ B (0, 1), $ in the last 
display on page 392 of \cite{W}, as follows from  the new version of  Lemma 3.18.  Second   argue as 
in Wolff to get  the top  display  on page 393  of his paper  with  $ \ep^{\si/(p-1)}$  replaced by
   $  \ep^{\si/(p-1)} N^{1-1/p} $.  Using this  display one gets the second display from the 
   top on page 393 with 
$ \ep^{ \si p'} $ replaced by  $  N^{(p-1)}   \ep^{ \si p'} $  and $f, v, g$ 
replaced by $  \bar v,  \ti \zeta_i, k_i, $  respectively. 
To  get the  next display  choose   $ 0 <  \ep' ,  $ in addition to the 
above requirements,    so that   
\begin{align} 
\label{3.64}   
| \nabla \bar v  +  \ep \nabla  \ti  \ze_i  |^{p-2}   \geq   \hat C^{-1}  N^{p-2}
\end{align}  
for  $  N  \geq N' ,  $   $  0 <  \ep  \leq \ep' , $  and $ 1 - \frac{y_0}{2N} \, \leq r \leq 1. $     
   This choice is  possible as we see from  \eqref{3.35}, \eqref{3.36} $(\be),$   and  \eqref{3.54}.    
   We can now  estimate the integral  in  \eqref{3.63},  using  Schwarz's    inequality and  \eqref{3.64} as in \cite{W}. 
 We  get   the  conclusion of   Lemma 3.19   in  Wolff's  paper \cite{W},   except the integral in this display is   now taken over  
 $   B ( 0, 1)  \sem B ( 0, 1 -  {\ts \frac{y_0}{2N}} ) . $  Now    \eqref{3.63},  \eqref{3.56} $(+)$,  and  the  fact  that   
$   \bar  v  $  has average 0  on  circles with  center at the origin,  are  easily  seen to  imply   as  in  \cite{W} that         
\begin{align} 
\label{3.65}  
\int_{- \pi }^{\pi}  k_1( (1-  {\ts \frac{y_0}{2N}}) e^{i\he} )  d \he - \int_{- \pi }^{\pi}  k_1( (1- {\ts \frac{ y_0}{4N})} e^{i\he} ) 
  d \he    \geq   C^{-1} \ep  
\end{align}  
provided   $ 0 <  \ep   \leq \ep'  $  and  $ \ep' $  is  small  enough. 
   From the triangle inequality we conclude that there is a 
   $ d  \in \{1 -  \frac{y_0}{2N},   1 -  \frac{y_0}{4N} \}, $  
   for which   if       $   \ti V ( z )  =   k_1 ( d z ) - k_1 (0 )$ for $z \in B (0, 1),  $  
 then either  $ V = \ti V $  or  $  V  =  - \ti V $  satisfies \eqref{1.3} $(c)$ in 
 Theorem \ref{thmA}.    Also  the usual  calculus of  variations argument giving  
 $ k_1  $  and the maximum principle for $ p $-harmonic functions,  as well as either
   \eqref{3.54} or  \eqref{3.55} and  \eqref{3.36} $(\be)$,   give  \eqref{1.3} $(a), (b)$ in 
   Theorem \ref{thmA} with $ c $  replaced by $ C$.  Finally  \eqref{1.3} $(d)$  of   
   Theorem \ref{thmA}  follows from these inequalities   and   Lemma \ref{lem2.3}.   
The proof  of  Theorem \ref{thmA}  is now complete for $2<p<\infty$.    

To avoid  confusion we prove  Theorem \ref{thmA}, for  $ 1 < p' < 2, $  rather than  
$ 1 < p < 2, $  where as usual $ p'  = p/ (p -1)  $  and  $ p >  2. $  To do this we  
first  replace  the  right-hand side  in  display (13) of  \cite{L1}  by   $ C \ep^{\ti\tau}  N^{p-1}, $  
as  we deduce  in view of   the   new  second  display   from the top on page 393 of  \cite{W}. 
Second we use (13)   and  Schwarz's  inequality in the second line of   display (12) in  \cite{L1}   
(with  $ Q $  replaced by $ B ( 0, 1 ) \sem B ( 0, 1 - \frac{y_0}{2N}),$   $ q = p$), 
and   either  \eqref{3.54} or \eqref{3.55}  to get      
\begin{align} 
\label{3.66} 
\begin{split}
 &{\ds \left| \int_{B ( 0, 1 ) \sem B ( 0, 1 - \frac{y_0}{2N})} r^{-2} \left[ |\nabla k_2 |^{(p-2)} (k_{2})_{\he}  - | \nabla \bar v + \ep \nabla \ti \ze_2 |^{(p-2)} ( \bar v +  \ep  \ti \ze_2 )_{\he}  \right] dx dy \right|^2 } \\
 & \hs{1.3in} \leq  C N^{2 (p-2) } \ep^{2 \ti \tau},  
 \end{split}
 \end{align} 
 where $ \ti \tau $ is as in    \eqref{3.63}.
     Taking  square roots  in   \eqref{3.66},  using  \eqref{3.60},  the fact that  
     $ |\nabla \bar v |^{p-2}  \bar v_{\he} $  has average  0 on circles with center 
     at the origin,    and    arguing as in  \cite{L1}  we get    
\begin{align} 
\label{3.67}   
\int_{ B (0, 1 -  \frac{y_0}{4 N} ) \sem B (0, 1 -  \frac{y_0}{2 N} )   } r^{-1}  | \nabla k_2 |^{p-2} ( r e^{i\he} )    (k_2)_{\he} ( r e^{i\he} )    dr  d  \he  \geq 
  C^{-1} N^{p-2}  \ep
\end{align}  
    for    $  N  \geq  N' $  and  $ 0 < \ep  \leq \ep' .  $     Let  $ k $  be the 
    $ p' $-harmonic function in  $  B ( 0,  1 )  $   with   $k(0)=0$ satisfying
\[ 
  k_r  = N^{2-p}      r^{-1} |\nabla( k_2)|^{p-2} (k_2)_{\he} \quad \mbox{and} \quad r^{-1} k_{\he}  =  -  N^{2 - p}  |\nabla   k_2|^{p-2}  (k_2)_{r}.   
  \]  
  Existence of $k$ follows from simple connectivity of $ B (0, 1) $ and the 
  usual existence theorem for exact differentials. 
 Then  \eqref{3.67}  implies   
\begin{align} 
\label{3.68}   
\int_{ - \pi }^{\pi}     k ( 1 - {\ts  \frac{y_0}{4 N}}  e^{i\he} )  d  \he  -    \int_{ - \pi }^{\pi}     k ( 1 - {\ts  \frac{y_0}{2 N}} e^{i\he} )  d  \he  \geq  
  C^{-1} \ep   
\end{align}  
  for   $ N \geq N' $  and $  0 < \ep \leq \ep'.$  Finally  \eqref{3.68}  and  a similar  
  argument  to the one   from  \eqref{3.65}  on in the first  case considered,  
  give  Theorem  \ref{thmA}  for $ 1 < p'  < 2. $      This completes the proof of Theorem \ref{thmA} for $1<p\neq 2<\infty$.

\section{Proof of Theorem \ref{thmB}} 
\label{sec4}  
\setcounter{equation}{0} 
 \setcounter{theorem}{0} 
In this section we first state   Wolff's  main lemma  for applications (Lemma 1.6 in  \cite{W}),    
in the  unit  disk  setting  and  then  use it  to  prove  Theorem  \ref{thmB}.   
The  proof  of Theorem \ref{thmB}  is   essentially unchanged from  Wolff's proof of  Theorem  \ref{thm1.1}. 
However for the readers  convenience  we outline his proof, indicating how  
to  resolve a   few   problems  in converting this proof from  a half space to  $ B (0, 1).$  
  We  also  note  that  if    $  V $ as  in Theorem \ref{thmA}  is $ 2\pi/ N  $  
  periodic in the $ \he $  variable,  where $ N  = k N_0, $  $ k = 1, 2, \dots, $    
  then  $ V $  is  $ 2 \pi/k $ periodic in this variable.  Also since  $ N_0 $   
  depends  only on $p$  in the wider context  discussed  below the statement of  
  Theorem \ref{thmA} in  section  \ref{sec1},      we may as well   assume   $ N_0 = 1. $     
  Finally in the proof of  Theorem \ref{thmB},  we let   $ c \geq 1,  $    
  denote  a  positive constant  depending only on $ p$  in this  wider context.

  \subsection{Main Lemma for applications  of   Theorem \ref{thmA}}        
Given  $ h  \in  W^{1,p} ( B ( 0, 1) ), $ let $  \hat h $  be  the  
$p$-harmonic  function  in  $ B ( 0, 1 ) $  with  boundary values  $ \hat h = h $  on 
$  \ar B ( 0, 1 ) $   in the 
$ W^{1,p} ( B (0, 1 ) )  $  Sobolev sense.  We  also let  $  \|  h  \breve \| $   
denote the Lipschitz norm  of     $ h $ restricted to  $ \ar  B ( 0, 1 ) $  
and  $ \|  h \|_{\infty}  =  {\ds \max_{ \ar B ( 0, 1)} | h |}. $  
Next we state an analogue   Lemma 1.6 in \cite{W}. 

\begin{lemma} 
\label{lem4.1}  
Let $ 1 < p < \infty. $ Define $ \al = 1 - 2/p$  if  $ p \geq 2, $ and $ \al = 1 - p/2, $ if $ p < 2.$  
Let $ \ep > 0$ and  $0<M < \infty.$  Then there are $ A = A ( p, \ep, M ) > 0  $
and $ \nu_0 = \nu_0 ( \ep, p, M) < \infty, $  such that if  $  \nu  >   \nu_0 \geq 1  $ is  an integer, 
$ f$, $g$, and $q $  are  periodic  on  $ \ar B (0, 1 )$  in the $ \he $ variable with periods,   
$2\pi$, $2\pi$, and $2\pi \nu^{-1}$, respectively and if      
\begin{align}
 \label{4.1}   
 \max ( \| f \|_{\infty}, \| g \|_{\infty}, \| q \|_{\infty},  \| f \breve \|, \| g \breve \|,  \nu^{-1} \| q \breve \| )   \leq M, 
\end{align} 
   then for $ z = r e^{i \he} \in B ( 0, 1), $  
\begin{align}  
\label{4.2}   
|  \widehat {q f +  g} ( r e^{i \he} ) -  f ( e^{i\he})  \hat q (r e^{i \he} ) - g ( e^{i \he} ) |   <  \ep  \quad \mbox{for}\, \,   1 -  r  < A \nu^{ - \al}.   
\end{align}
   If, in addition,  $  \hat q (0)  = 0, $  then  
\begin{align}
\label{4.3} 
|  \widehat{ q f + g } \, ( ( 1 - A \nu^{- \al} )  e^{i\he} )  -   g ( e^{i\he}) | < \ep    
\end{align}    
and   
\begin{align} 
\label{4.4} |  \widehat{ q f + g } ( r e^{i\he} )  -   \hat g (r e^{i\he}) | < \ep    \quad \mbox{if} \, \, r  <   1 -  A  \nu ^{-\al}.
\end{align}      
\end{lemma}

\begin{proof}   Lemma \ref{lem4.1}  is  just  a  restatement  for  $ B (0, 1), $  of  Lemma 1.6
   in   \cite{W}.  To briefly outline the proof  of   Lemma \ref{lem4.1},  we note that   Lemma 1.4 in \cite{W}
     is used to prove    Lemma   1.6  in  \cite{W}. This lemma relative to  $ B (0, 1) $  
     states for fixed $ p$, $1 < p < \infty, $  that if   $ u$ and $v $ are  $ p $-harmonic in $ B (0, 1),$  
     bounded,   $ u, v   \in  W^{1, p}  ( B(0, 1) ), $    and  if    $ u \leq    v $ on  
     $  \{ e^{i\he} :  | \he  - \he_0 |  \leq  2 \eta \}$ for $0  < \eta < 1/4$   
     in the $ W^{1,p} ( B ( 0, 1) )$ Sobolev sense,   then for   $ 0 < t \leq 1/2, $     
\begin{align} 
\label{4.5}   
\int_{1-t}^1 \int_{\he_0 - \eta}^{\he_0 + \eta}  | \nabla ( u  - v  )^{+}   | \,  r dr d \he   \leq   c \eta^{-1}  t^{1/p'} ( \| | \nabla  u |  \|_p  +  \| | \nabla v | \|_p )^{\al } \, \,   [  \max_{\ar B ( 0, 1)}  ( u  - v )^{+} ]^{1- \al}
\end{align}
    where $  a^+  =  \max (a, 0) $. It follows  from   a Caccioppoli  type  inequality for   $  ( u - v)^{+}$ that \eqref{4.5} holds.

To   begin the    proof  of  Lemma \ref{lem4.1},   if $ z = r e^{i\he} \in B ( 0, 1), $  let   
\[      
J ( r e^{i\he} )  =    \widehat {q f +  g} ( r e^{i \he} )  -   \hat q (r e^{i \he} )  f ( e^{i\he})   -   g ( e^{i \he} ) .      
\]  
The first step in the proof of  Lemma \ref{lem4.1} is to show for given $  \be  \in 
(0, 10^{-5})  $ that   there is  a  $   A  =  A ( p, \ep,  M,  \be  ) $  for which  
\eqref{4.4} holds (so $ |J| (t e^{i\he})  < \ep$)    when   $  \be \nu^{-1} < 
     1 - t  <   A  \nu^{-\al},   $  for  $  \nu   \geq  \nu_0 = \nu_0 ( p, \ep,  M,  \be). $   
     Indeed if   $  J ( t  e^{i \he_0 } )   >   \ep, $   then \eqref{4.1},  
     Lemmas \ref{lem2.1},   \ref{lem2.2},  and invariance of  $p$-harmonic 
     functions under a  rotation,    are used   in \cite{W} to show that if    
     $  \eta =  \frac{\ep}{ 10^5 (M^2 + M)}, $  then  there is  a  set  
     $ W \subset  \{ t  e^{i\he} : | \he - \he_0 |      \leq \eta  \}   $  of   
     Lebesgue measure  $ \de \geq    \rho \be \eta/100,  $  where $ \rho = \rho ( p, M, \ep) , $   with  
\begin{align} 
\label{4.6}   \widehat {q f +  g} ( t e^{i \he} )  -    \hat q (t e^{i \he} )  f ( e^{i\he_0})  -   g ( e^{i \he_0} )
     >   \ep/2 . 
\end{align}     
         Also  \eqref{4.1}  and the choice of  $ \eta $  yield  
\begin{align}  
\label{4.7}   
| \widehat {q f +  g} (  e^{i \he} )  -    \hat q ( e^{i \he} )  f ( e^{i\he_0})  -   g ( e^{i \he_0} )| < \ep/200   \quad \mbox{when}\, \,  | \he  - \he_0 |  <  \eta.
\end{align}
Using    \eqref{4.7}, our  knowledge of $ W,  $  and     \eqref{4.5}   it  follows that if  
$ u (r e^{i\he})  =  \widehat{ q f + g} (r e^{i\he})  $  and  
$ v ( r e^{i\he} )   = q (r e^{i\he})  f ( e^{i\he_0} ) + g ( e^{ i \he_0} ), $  then    
\begin{align} 
 \label{4.8}   
\begin{split} 
 \de  \ep/4   &\leq  {\ds \int_0^t \int_{\he_0 - \eta}^{\he_0 + \eta}  | \nabla ( u  - v  )^{+}   | \,  r dr d \he }   \\
 & \leq   c (M, \ep )  t^{1/p'} ( \| | \nabla  u |  \|_p  +  \| | \nabla v | \|_p )^{\al} \, \,  \\ 
 &   \leq    c'(M, \ep ) t^{1/p'}  \nu^{  \al/p'}.
 \end{split}
 \end{align}
     The  estimate on   $ \|  | \nabla u | \|_p$ and $\| |\nabla v| \|_p,  $  in the second
        line  of \eqref{4.8}   follows from  \eqref{4.1} and  the minimization property
         of  $p$-harmonic functions using, for example,
\[   
\psi  ( r e^{i \he}  )   = u( e^{i\he})  \chi ( r ) \quad \mbox{where} \quad  \chi \in  C_0^{\infty} ( 1 - 2/\nu,  1 + 2/\nu)
\]
 with  $ \psi = 1$  on  $( 1 - 1/\nu, 1 + 1/\nu ) $  and  $ |\nabla \psi | \leq c \nu. $    Now \eqref{4.8} 
 yields after some arithmetic that  $ t >  \ti A ( \ep, M, \be ) \nu^{ -  \al   }. $   Thus   \eqref{4.2} of Lemma  
\ref{lem4.1}  is true  when 
$  \be \nu^{-1}  <  1 - r   <  A  \nu^{- \al } , $  subject to fixing $ \be = \be ( \ep, M). $   
To do this  we apply   \eqref{2.3a} of    
Lemma  \ref{lem2.2}    with   $ \hat v = \widehat{qf+g}$,  $ q,  $  and 
with $ \rho =  \be^{1/2} \nu^{-1}$, $\si = 1$,  $M'  =  \nu$,   to get   
for $ 1 - r  <  \be \nu^{-1},$       
\begin{align} 
\label{4.9}  
  | J ( r e^{i \he} )    -  J ( e^{i\he} ) |  \leq    c ( M ) \left(  \nu \,  ( \be^{1/2} \nu^{-1}    )    +     ( \frac{ \be \nu^{-1}  }{\be^{1/2}  \nu^{-1}}   )^{\si_1}\right)   \leq  c' ( M ) \,   \be^{\si_1/2} .  
  \end{align} 
 Choosing $ \be = \be ( \ep, M ) > 0 $ small enough and then fixing  $ \be $ we obtain   \eqref{4.2} for     
$ 1 - r  <  \be \nu^{-1}.$         

To prove  \eqref{4.3} we note from \eqref{2.8} of  Lemma  \ref{lem2.6} that   
\begin{align} 
\label{4.10}   
| q ( r e^{i\he}) -  q (0)  |  \leq  c M    r^{\nu/c}
\end{align}    
where  $ c = c(p).$   Using  \eqref{4.10} with  $ q(0) = 0$,   $r  =  1 - A \nu^{-\al}, $ and  
choosing $ \nu_0, $  still larger if necessary  we get  \eqref{4.3}. Now  \eqref{4.4} 
follows from  \eqref{4.3} and  \eqref{2.3a} of Lemma \ref{lem2.2}  with  $ \hat v = \hat g$ and  $\rho =
 A \nu^{- \al/2}   $  in the same way as  in the  proof of    \eqref{4.9} for $ \nu_0 $ large enough.    This finishes the sketch of proof of Lemma \ref{lem4.1}. 
\end{proof}

\subsection{Lemmas   on  Gap Series}  
The  examples  in  Theorem  $ \ref{thmB} $ will be  constructed  using   Theorem    \ref{thmA} 
as   the  uniform limit on compact subsets of  $  B (0, 1)  $ of a sequence of  $p$-harmonic  functions
  in  $  B (0, 1), $     whose boundary values are  partial sums     of       $  \Ph_j $ in  Theorem \ref{thmB}
   with  periods  $ 2\pi/ N_j $   where  $ N_{j+1}/N_j  > > 1$.   Lemma \ref{lem4.1}  will be used to make estimates on this sequence.   
   Throughout this subsection we let $ |E|  $  denote the Lebesgue measure of  a  measurable set  $  E \subset \re. $    We  begin with  
\begin{lemma} 
\label{lem4.2} 
For $j = 1, 2,  \dots,$  let  $  \psi_j$  be   Lipschitz functions defined  on  $  \ar B  ( 0, 1) $  with    
\begin{align} 
\label{4.11}    
\int_{-\pi}^{\pi}  \psi_j ( e^{i\he} ) d \he  =  0  \quad \mbox{and}\quad      \| \psi_j   \|_{\infty}  +  \| \psi_j  \breve \|     \leq    C_1 < \infty.
\end{align}
For $j = 1, 2, \dots$, let  $  (N_j)_1^{\infty}  $ be  a  sequence of  positive integers   with  
$ N_{j+1}/N_j  \geq  2$.  Also let  $ (a_j)_1^{\infty} $ be a sequence of real numbers with 
$ {\ds \sum_{j=1}^\infty  a_j^2  <  \infty} . $  

 If        
  \[ 
  s^* (  e^{i\he} ) : =  \sup\limits_{k}   \left|   \sum_{ j = 1}^k   a_j  \psi_j ( e^{ i N_j \he} )  \right|
  \]
 then        
\begin{align} 
\label{4.12}  
\int_{- \pi}^{\pi}  ( s^*)^2 ( e^{i\he} )  d \he  \leq  c \, C_1^2  \sum_{j =1}^\infty a_j^2 
\end{align}
  where $ c $  is  an  absolute constant. Consequently, 
\begin{align}
\label{4.13}  
\begin{split}
&(a) \hs{.2in} s ( e^{i\he} ) : =  {\ds \lim_{k\to \infty}  \sum_{j=1}^k }  a_j   \psi_j ( e^{ i N_j \he} ) \quad \mbox{exists for  almost every}\, \,   \he  \in  [- \pi, \pi], \\
&(b) \hs{.2in}             | \{  \he \in [- \pi , \pi] :   s^* ( e^{i\he} ) > \la \} |   \leq     \frac{ c \, C_1^2}{\la^2}. 
\end{split}
\end{align}
\end{lemma}

\begin{proof}   
Using  elementary properties of  Fourier series  
(see  \cite{Z}) and  $ \| |\frac{d\psi_j}{d\he} | \|_{\infty}     \leq  C_1 $  we  find  that    
\begin{align} 
\label{4.14}      
\psi_j ( e^{i\he} )  =  \sum_{ n = - \infty}^{\infty}  b_{j n } e^{ i  n  \he}  \mbox{ where }  b_{j 0} = 0  \quad \mbox{and}\quad   \sum_{ n = - \infty}^\infty   n^2  b_{j n}^2  \leq  c\,  C_1^2.
\end{align}
Now     
\begin{align}  
\label{4.15}  
s^* (e^{i\he} )  \leq    \sum_{n=-\infty}^{ \infty}  \sup\limits_k  \left|   
\sum_{j=1}^k  a_j  b_{j n} e^{i n N_j \he }  \right| =    \sum_{n=-\infty}^{ \infty} l_n^* ( e^{i\he}) 
\end{align}  
where $ l_n^* $ is the maximal function of   $  {\ds \sum_{j=1}^\infty   a_j  b_{j n} e^{i n N_j \he }.}$    It is well known  (see \cite{Z}) that  
\begin{align}
\label{4.16}  \int_{-\pi}^{\pi} (l_n^*)^2  (e^{i\he}) d \he   \leq \, c' \,  \sum_{j=1}^{\infty} (a_j b_{jn})^2.
\end{align}     
Using  \eqref{4.15}, \eqref{4.16}, and Cauchy's inequality we  get   
\begin{align}
\label{4.17}   
\begin{split}
 \int_{-\pi}^{\pi}  (s^*)^2 ( e^{i\he} ) d \he  &\leq      \left(  \sum_{n=-\infty}^\infty ( \int_{-\pi}^{\pi}  (l_n^*)^2 ( e^{i\he} ) d \he)^{1/2}  \right)^2
\leq     {\ds   c' \left( \sum_{n=-\infty}^\infty   ( \sum_{j=1}^\infty  a_j^2  b_{j n}^2 )^{1/2}  \right)^2  } \\ 
&\leq  {\ds  2 c' \,  ( \sum_{n=1}^\infty n^{-2}    ) \sum_{j=1}^\infty  \sum_{n=-\infty}^{\infty} } (a_j \, n  b_{j n} )^2 \leq  c \, C_1^2  {\ds  \sum_{j=0}^\infty  a_j^2 }.
\end{split}
\end{align}
Therefore, \eqref{4.12} is valid. Now  \eqref{4.13} follows from   standard  arguments, using  \eqref{4.12} (see  \cite{Z}).       

To prove Theorem  \ref{thmB}, let  $ N_j  $  be a  sequence of positive integers 
with  $ N_{j+1}/N_j $  a  positive integer $ > 2.$    Let  $  \Ph_j $  be  
the  $p$-harmonic function in  Theorem \ref{thmA} with period  $  2 \pi/ N_j $   and set 
$   \ti \Ph_j   =  \frac{\Ph_j}{\| \Ph_j\|_{\infty}}.$   Also for $  \he \in \re  $    and   
$ j = 1, 2, \dots $, we set 
\begin{align} 
\label{4.17a}    
\begin{split}
&\ph_j ( e^{i\he} ) = \ti \Ph_j ( e^{i\he/N_j}),  \\
& d_j    =    \frac{1}{2\pi}  \int_{-\pi}^{\pi}  \ti  \Ph_j ( e^{i\he} )  d \he, \\
&   \psi_j  =  \ph_j -  d_j.
\end{split}
\end{align}
Note from Theorem \ref{thmA}  that   $ c^{-1}  \leq d_j  \leq  1$  and that $\psi_j$   
satisfies  \eqref{4.11} of    Lemma    \ref{lem4.2} for $j = 1, 2,  \dots,  $.  For $j = 1, 2,  \dots,  $ set  
\[
G _j  :=  \{  [ \pi k/ N_j,  \pi (k + 2)/N_j ], \, \,  \mbox{$k$ an integer} \}    
\]
and let $   \{ L_j ( e^{i\he} ) \} $  be continuous functions on $  \ar B (0, 1 ) $ satisfying          
\begin{align}
\label{4.18} 
\begin{split}
\begin{cases}
L_1 \equiv  1, \, \, 0 < L_ {j+1} \leq L_j, \, \, \mbox{and} \\
L_{j+1}/L_j\, \,  \mbox{considered as  a function of  $  \he $ on  $ \re  $  is linear on the  intervals in $ G_j$}.
\end{cases}
\end{split}
\end{align}
Let $  \si ( e^{i\he}  )$ for  $\he \in \re,  $  be the formal series defined by
\begin{align} 
\label{4.19}   
\si (e^{i\he} ) :=  R +  \sum_{j=1}^\infty  a_j \,  L_j ( e^{i\he} ) \,  \ti \Ph_j  ( e^{ i \he} ) 
\quad \mbox{and}\quad   \ti s ( e^{i\he} ) :=  \sum_{j = 1}^\infty  a_j  \ti \Ph_j  ( e^{i\he})
\end{align}
with
\begin{align}
 \label{4.20} 
  0  \leq  |R|  \leq 1  \quad \mbox{and}  \quad  \sum_{j=1}^\infty a_j^2  <  1  \, .  
\end{align} 
Finally let  
$ \ti s_n$ and $\si_n$  denote corresponding  $n$-th partial sums  of  $  \ti s$ and $\si$ respectively.   

Given  $ I   \in G_j $, let $ \ti  I $  denote the interval with the same center as $ I $  
and three times its length.   Using  the gap assumption on $(N_j), $  \eqref{4.18}, and induction  we find  that 
\begin{align}
  \label{4.21} \| L_{j+1} \breve \|   \leq  c  N_j.
 \end{align}
Using the gap assumption on $ (N_j ), $  Theorem  \ref{thmA},   and  \eqref{4.21},  \eqref{4.20}, we deduce for  $ n = 1, 2, \dots, $  that 
 \begin{align} 
 \label{4.22}  \be_n  =  N_n^{-1}  ( \| s_n \breve \| +  \| \si_n \breve \| ) \leq  c 
 \quad \mbox{and} \quad  \lim_{n \to \infty} \be_n =  0  \mbox{ as }  n \to \infty. 
\end{align} 
   Moreover, from \eqref{4.13}$(b)$ of    Lemma \ref{lem4.2} and  \eqref{4.17a}  we have   
\begin{align} 
\label{4.23}   
|  \{ \he  \in [- \pi, \pi ] :    \sup\limits_n  |  \ti s_n ( e^{i \he } )   -   \sum_{j = 1}^n  d_j  a_j |   > \la  \}  |  \leq  c  \,  \la^{ - 2 }  \sum_{j=1}^{\infty} a_j^2  .  
\end{align}
    First let $ R = 0 $  and choose  $ (a_n) $  satisfying 
 \eqref{4.20},   so   that   
$  \sum_{ j = 1}^{\infty}  d_j a_j  $  is  a  divergent series  whose partial sums are  bounded.   Then  from \eqref{4.23} we deduce that 
\begin{align} 
\label{4.24}  
\sup\limits_n   | \ti s_n ( e^{i\he} ) |  <  \infty  \quad \mbox{and} \quad     \ti s ( e^{i\he} )  \mbox{ does not exist  for almost every }  \he  \in [ - \pi, \pi ]. 
\end{align}
  Using   \eqref{4.18}-\eqref{4.24},    Wolff (see \cite[Lemma 2.12]{W}) essentially proves 
  \begin{lemma}  
  \label{lem4.3}   
If   $ N_{j+1}  >  N_j  ( \log ( 2 + N_j ) )^3 $ for $j  = 1, 2, \dots,  $ then there is  a  choice of 
$ ( L_j ) $ satisfying     \eqref{4.18}   such that   $   \sup\limits_{j}  \|  \si_j \|_{\infty} < \infty  $  and  $  \si $  
diverges for almost every  $  \he  \in [-\pi, \pi]. $   
\end{lemma}

\begin{proof}   
To  outline the proof of  this lemma, for $n  = 1, 2, \dots,$ let   $ \Upsilon_n$ denote all intervals $ I \subset  \re $  that are maximal
 (in length) with the property   that  $ I \in   G_j $  for some $j$ and   $ {\ds  \max_{I} } \,  |s_j|  > n. $     
From   \eqref{4.21},   \eqref{4.22},  and   \eqref{4.20}   we  see that  
if   $ I  \in  \Upsilon_n \cap G_j , $ and  $ \ti c $  is  large enough 
(depending only on $p$),   then  $ |s_j| > n  -  \ti c  $ on $ \ti I $ 
where $ \ti c $ depends only on $p.$    Using \eqref{4.23}   with  $  \la = n   - \ti c $  
and boundedness of the partial sums of   $  \sum_{j=1}^{n} a_j d_j $      
we get   $ c  \geq 1 $  depending only on  $ p, $  and the  choice of  
$ (a_j) $    such that   
\begin{align}
\label{4.25}    
\sum_{ \ti I  \in \Upsilon_n}    |  \ti I   \cap [- 3 \pi, 3 \pi]  |    \leq  c \, n^{-2} \quad  \mbox{for} \,\,  n = 1, 2, \dots.
\end{align}
Thus,  from  the usual measure theory argument,  
\begin{align}  
\label{4.26}  
\left|  \{ \he  \in  [-\pi, \pi]  :    \mbox{ for  infinitely many $ n$,   $\he \in   \ti I$  with $  I \in \Upsilon_n$\}   }   \right|  =  0.
\end{align} 

 Finally, for  $j = 2, \dots,$   define $ L_j$  by induction as follows 
\begin{align} 
\label{4.27}  
\begin{split}
&(a) \hs{.1in} \mbox{ If  $ L_k $ has  been defined  and  $ I  \in   G_{k} $ is also in  $ \cup  \Upsilon_n, $ put $  L_{k+1} = \frac{1}{2} L_k  $ on   $ I. $}  \\  
&(b) \hs{.1in}   \mbox{    If   none   of  the three intervals in  $ G_{k}  $  contained  in  $ \ti I  $   are  in   $  \cup  \Upsilon_n,  $}   \\
& \hs{.47in}  \mbox{ set    $  L_{k+1} =  L_k  $ on   $ I. $ }  \\  
&(c) \hs{.1in} \mbox{ If  neither  $ (a) $  nor  $(b)$  holds for $ I \in G_{k} $  use   \eqref{4.18} to define $ L_{k+1}. $ }
\end{split}
\end{align}
From    \eqref{4.26}  and the definition of  $ L_j $ we see for almost  every $ \he \in [ - \pi, \pi] $  
that  there exists a positive integer  $ m  = m ( \he) $  such that   
 $ L_j  ( \he)  = L_m ( \he )  $  for  $  j \geq m. $   From \eqref{4.24} 
 we conclude that   $ ( \si_k )  $ diverges for almost every $ \he  \in [ - \pi,  \pi]. $  
 Also  if    $   | \ti s_k (e^{i\he} ) |   >  n ,   $     then    since  $  |a_k | \| \ti \Ph_k \|_{\infty}  \leq  1, $ 
 we see from \eqref{4.27} that  there exist   $ n   $ distinct integers,
   $ j_1  < j_2 <  \dots  j_n  \leq k $  with $   L_{j_i + 1} (e^{i\he} ) = \frac{1}{2}  L_{j_i} ( e^{i\he} ) .   $   Thus  $ 
L_{k+1} ( e^{i\he} )  \leq 2^{ - n}. $  Using this fact and summing by parts Wolff gets,  
 $  \sup\limits_k  \| \si_k \|_{\infty}  <  \infty.  $   
\end{proof}  

Next we state 

\begin{lemma} 
\label{lem4.4}  
If   $ N_{j+1}  >  N_j  ( \log ( 2 + N_j ) )^3 $ for $j  = 1, 2, \dots,  $ then there is  a  choice of 
$ ( L_j ) $ satisfying     \eqref{4.18}   such that   $     \si_j    > 0$ for $j = 1, 2, \dots   $  and  $ \sup\limits_j \| \si_j \|_{\infty} < \infty$    on  $ \re. $   Also,   
\[
\si ( e^{i\he} ) ={ \ds \lim_{j \to \infty}  \si_j }  ( e^{i\he} ) = 0   \quad  \mbox{for almost every}\, \,  \he  \in [-\pi, \pi]. 
\] 
 \end{lemma} 
 
\begin{proof}  
Lemma \ref{lem4.4}     is   essentially   Lemma
 2.13   in \cite{W}.  To   outline  his  proof  let  
\begin{align} 
\label{4.28}   
R = 1   \mbox{ and }   a_j     =   -  \frac{1}{4j } \,  \mbox{   for $ \,   j =1,2, \dots  $  in     \eqref{4.20}. }   
\end{align}
We also set
\[
 \Upsilon_{kn}    :=  \{  I  \in G_k  :    \max_I   \ti s_k   > n   
\mbox{   and  $  I  \not  \subset J \in  \Upsilon_{jn}$   for any $ j < k$ } \}.
\]   
Define   $ \mathcal{F}_{kn}$ and $\mathcal{H}_{kn},$ by induction   as follows : Let 
 $  \si_1 =   1 +   a_1 \ti \Ph_1 $  be the first partial  sum of  $ \si  $  in  \eqref{4.19}.      
 By induction,  suppose    $ L_j $   
   and corresponding  $  \si_j $   have  been defined for  $ j  \leq  k. $  
   Assume also that  
   $ \mathcal{F}_{jn}, \mathcal{H}_{jn}  \subset  G_j $ have been 
   defined  for $ j < k $  and all positive integers $n$ 
  with $ \mathcal{F}_{0n} = \es  =   \mathcal{H}_{0n}.$      If  $ n $ is  a 
  positive integer and  $ I  \in G_k,  $  we  put  $ I  \in \mathcal{F}_{kn} $  
  if  $  \min_{I} \si_k  <  2^{-n} $  and  this interval is not in $  \mathcal{F}_{jn} $ 
  for some $  j < k.$    Moreover we put  $ I  \in  \mathcal{H}_{kn} $ 
  if     $  \min\limits_{I} \ti s_k  <  - \frac{ 2^{n}}{n+1}$  and 
  $ \max\limits_{I} L_k > 2^{-n}. $     Then   
\begin{align} 
\label{4.29}   
\begin{split}
&(a)  \hs{.2in}  L_{k+1} =  \frac{1}{2} L_k   \mbox{ on    $   I  \in G_k $  if  $ I  \in  \cup_n  ( \mathcal{F}_{kn}  \cup  \mathcal{H}_{kn}  \cup    \Upsilon_{kn}  )   $ }  \\   
&(b)  \hs{.2in}  L_{k+1} = L_k   \mbox{  on  $ I $  if  none of  the three intervals in  $ \ti I $ are  in     $ \cup_n  ( \mathcal{F}_{kn}  \cup  \mathcal{H}_{kn}  \cup    \Upsilon_{kn} ). $ } \\
& (c) \hs{.2in}  \mbox{If neither $ (a)$ nor $(b)$  hold for $ I \in G_k, $ use  \eqref{4.18} to define $  L_{k+1}. $ }  
\end{split}
\end{align}
This definition together with \eqref{4.19} define  $ L_{k + 1}$ on  $  G_k $   so by induction we get  $ (L_m )$, $(\si_m), $  and also  
  $  (\mathcal{F}_{mn})$ ,  $(\mathcal{H}_{mn})$, $(\Upsilon_{mn}) $ whenever  $  m, n  $ are positive integers.

  As in Lemma \ref{lem4.3} we have $ L_{k+1}  < 2^{-n} $ on   $ I  \in  \Upsilon_{kn}$.  Also if  
  $ 2^{-(n+1) }   \leq \min_I \si_j  <  2^{-n} $  and  $ L_{j+1} \leq 2^{-n} $  on $ I \in G_j, $ then from  \eqref{4.28}  we see that  
  \[  
  \si_{ j + 1}  \geq  \si_j   -  2^{-(n+2) }  \geq 2^{-(n+2)} \quad \mbox{on}\, \,  I.
   \]    
   Using   this   observation and induction  on  $ n  $ one can show  for all  positive integer $k$ and $n$ that    
\begin{align}  
   \label{4.30}  
   \mbox{  if  $ I  \in G_k $   and $ \min\limits_I \,  \si_k < 2^{ - n}$  then $ L_{k+1} < 2^{-n} $ on $ I \in G_k$}.  
\end{align}   
 Now    \eqref{4.30}   implies that 
\begin{align}  
\label{4.31}      
\si_k   > 0 \, \, \mbox{for}\, \, k = 1, 2, \dots,  \mbox{ on } [- \pi, \pi ]
\end{align}
 since  $ \si_1 > 0 $  and  if    $   2^{- ( n + 1) }  \leq   \si_{k} <   2^{-n} $  
 on $ I  \in G_k$.  Using this observation and \eqref{4.30} again we  have     
 \[
 \si_{ k + 1}  >   \si_k   -  2^{ - ( n + 2)} > 0   \quad  \mbox{on}\, \, I.
 \]    
Thus   to show that  $ (\si_j) $ is bounded it suffices to show that  $  \max\limits_{k}\,  \si_k <  c  < \infty. $  
Using this fact and  repeating the argument for boundedness of $ ( \si_j ) $ in   
Lemma \ref{lem4.3}  we obtain boundedness of    $  (\si_j) $.  It remains to prove that 
\begin{align} 
\label{4.32}  
\begin{split}
\ti s_k ( e^{i\he} )  \to 0 \quad  \mbox{for almost every}\, \,  \he \in [-\pi, \pi]. 
\end{split}
\end{align}
We shall need
\begin{align}
\label{4.32a} 
c^{-1}   \, L_{k+1} (e^{i\he_2} ) \leq  L_{k+1} ( e^{i\he_1} )  \leq  c  \, L_{k+1} (e^{i\he_2} ) 
\end{align}
whenever $\he_1,\he_2  \in \ti I$ and $I  \in G_k$  for $k = 1, 2 \dots$. 
This  follows easily from \eqref{4.29} and the gap assumption on  $ (N_j) .$ 
        To prove  \eqref{4.32}   let  $ E_n $  denote the  set of  all  $  \he \in \re $  
        for which there exist  $ k$ and $l  $  positive integers with 
        $ k  < l $  satisfying
\[
\ti s_l  >  -  \frac{ 2^n}{2(n+1)} \quad      \mbox{while} \quad  \ti s_k < -  \frac{ 2^n}{n+1}. 
\]
        From $ a_j  < 0 $ and   $ c^{-1}  \leq d_j \leq 1$ for $j = 1, 2, \dots$,  we obtain  that   
 \begin{align}  
 \label{4.33}   
 \max \left[  | \ti s_l ( e^{i\he})   -  \sum_{j=1}^l a_j d_j |\, , \, |  \ti s_k ( e^{i\he} ) -  \sum_{j=1}^k a_j d_j |\right ]     \geq   \frac{ 2^n}{ 8 (n+1)}
\end{align} 
 for $n\geq 100$. If we let 
  \[    
   \La  :=   \{ \he \in \re : \he \in E_n \mbox{ for infinitely many $ n $}  \} \cup 
  \{ \he \in \re  :  \limsup_{j\to \infty}  \ti s_j (e^{i\he})  > -  \infty \}
  \]
  then using   \eqref{4.33} and  \eqref{4.23} we deduce 
\begin{align} 
\label{4.34}     
|  \La  |  = 0.
\end{align}
Next from induction on  $ m  $   and   the definition of  $ \mathcal{H}_{km}, $  it follows    that
    if   $  \ti s_k ( e^{i\he} )  <  -  \frac{2^m}{m+1}$ on $ I \in G_k$   
    then   $ L_{k+1 } ( e^{i\he} )  \leq 2^{-m}. $  Therefore if  $\he_0  \not \in \La$  
    then 
\[    
    \lim_{k \to \infty} \ti s_k \,  (e^{i\he_0} ) = - \infty  \quad \mbox{and} \quad \lim_{k \to \infty} (\ti s_k \,  L_{k +1} ) (e^{i\he_0} ) = 0. 
\]    
 
     These    equalities  and  
\begin{align} 
\label{4.34a}   
0  \leq  \si_j   ( e^{i \he_0})   = 1 +  \sum_{l \leq  j }  ( L_l  - L_{l+1} ) \ti  s_l  ( \he_0 )   +   s_{j} L_{j+1} ( \he_0)
\end{align}
 imply  that if    $ \he_0  \not \in \La $  then  it must be true that $ \lim_{j \to \infty } \si_j (e^{i\he_0} )$ exists and is non-negative.   

 Suppose this limit is positive.  Then from \eqref{4.22}  we find that  
\begin{align}  
\label{4.35}  
\sup\limits_{j} \{  \max_{\ti I }  \ti s_j  : \he_0 \in I  \in G_j \}   < \infty \quad \mbox{and} \quad
   \inf\limits_{j} \{   \min_{\ti I} \si_j  : \he_0 \in I \in G_j \}   >  0 . 
\end{align}   
      So   $  \he_0 $ belongs to at most a finite number of  $ \ti I $  with   
$ I \in \Upsilon_{kn}   \cup \mathcal{F}_{k, n}  $  for  $ k, n = 1, 2, \dots $. Then since $ L_k ( e^{i\he_0} )  \to   
   0$ and $\ti s_k ( e^{i\he_0} )  \to - \infty  $ as $ k \to \infty $  
   we deduce that given  $ m  $  a  sufficiently large positive integer, 
   say  $ m \geq m_0, $  there exists  $ m'  < m $  with   
\[
    m' = \max\{ j : j< m \mbox{ and }  L_j (\he_0)\neq   L_{m} ( \he_0 ) \} 
\] 
such that  $ \he_0 \in \ti I$,    $I  \in \mathcal{H}_{m'n},   $ 
for some  positive integer $ n $.  This  inequality  and   \eqref{4.32a}  yield  
that  if   $ L_{k} ( \he_0 ) <  2^{-l}, $ then  $ \ti s_k ( \he_0 ) 
<  - c  \, \frac{ 2^{l}}{ l + 1}  $ for $  l \geq l_0 $  where $ c \geq 1  $ is 
independent of  $k$ and $l$.    Using this fact and choosing an increasing  
sequence $ (i_l)   $ for $ l \geq l_0 $  so that 
$ L_{i_l} ( \he_0)  = 2^{ - l }$ for $l \geq l_0,   $  it  follows from \eqref{4.34a} that 
$ \si ( e^{i\he_0} ) = - \infty $  which contradicts  $  \eqref{4.31}$.
This first shows that $\si ( e^{i\he_0} )=0$ and this completes the proof of Lemma \ref{lem4.4}. 
\end{proof}   
\subsection{Construction of Examples} 
        To  finish the  proof  of  Theorem \ref{thmB}  we again follow   Wolff in \cite{W} closely  and use  
        Lemmas  \ref{lem4.1}, \ref{lem4.3}, and  \ref{lem4.4} to construct examples.    
        Let  $ N_1 = 1$ and  by  induction suppose  $ N_2,  \dots, N_k $ have been chosen,    
        as in  Lemmas  \ref{lem4.3} and \ref{lem4.4}, with  $ \si $  as in  \eqref{4.18}-\eqref{4.20}.     
        Let $ g = \si_k$, $f = a_{k+1} L_{k+1}$,    $ q   = \Ph_{k+1} ,   $  and suppose     
\begin{align} 
\label{4.36}   
\max ( \| f \|_{\infty}, \| g \|_{\infty}, \| q \|_{\infty},  \| f \breve \|, \| g \breve \|, N_{k+1}^{-1} \| q \breve \| )   \leq M 
\end{align}
  where $ M = M ( N_1, \dots, N_{k}) $ is a constant  and  $  \Ph_{k+1} $ is  $p$-harmonic in  $ B (0, 1) $ 
  with Lipschitz  continuous boundary values  and  $ \Ph_{k+1} (0) = 0. $  
  Next   apply  Lemma \ref{lem4.1}  with  $ M $ as in 
    \eqref{4.36} and  $  \ep = 2^{-(k+1)} $ obtaining $ A = A_k $  and $  \nu_0 $ 
    so that \eqref{4.1}-\eqref{4.4} are valid.  We also choose  $ N_{k+1} > \nu_0 $ 
    and so that $ A_{k} N_{k+1}^{- \al }  <  \frac{1}{2} A_{k-1} N_k^{- \al}  $      
        where $  \al = 1 - p/2 $ if $ p < 2$ and $ \al = 1  - 2/p$ if $ p > 2.$   
        By induction we now get  $ \si $ as in 
        Lemma \ref{lem4.3} or  Lemma \ref{lem4.4}. Then                      
\begin{align} 
\label{4.37}   
| \hat \si_{j+1}  ( r e^{i\he} )  -  \hat \si_j ( r e^{ i \he} ) | < 2^{ - (j+1) }  
      \quad  \mbox{ when }  r < 1 - A_{j} N_{j+1}^{-\al},  \,        
\end{align}  
and      
\begin{align}
\label{4.38}   
| \hat \si_{j+1}  ( r e^{i\he} )  -  \si_j (  e^{ i \he} ) | < 2^{ - j }  + |a_{j+1}| \quad   \mbox{ when }  r  > 1 - A_{j} N_{j+1}^{-\al} \,  .     
\end{align}
  From   \eqref{4.37}   we see that $ (\hat \si_{j+1}) $  converges  uniformly  on compact 
  subsets of $ B (0, 1) $   to  a  $p$-harmonic  function  $  \ti   \si$   satisfying  
\begin{align}
 \label{4.39}   
 | \ti \si  ( r e^{i\he} )  -  \hat \si_k ( r e^{ i \he} ) | < 2^{ - k }      \quad    \mbox{ when }  r < 1 - A_{k} N_{k+1}^{-\al}.  
\end{align}
Using  \eqref{4.37}, \eqref{4.39},   and the triangle inequality we also have  
for $ 1 -  A_{k}  N_{k+1}^{-\al} < r  <     1 -  A_{k+1}  N_{k+2}^{-\al}    $    that      
\begin{align}
\label{4.40}   
\begin{split}
|\ti \si (r e^{i\he} )  -   \si_k  ( e^{ i \he})|  &\leq     | \ti \si (re^{i\he})  -   \hat \si_{k+1} ( r e^{i\he})|   +  |   \hat \si_{k+1} ( r e^{i\he} )   -   \si_k  (  e^{ i \he} ) |  \\ 
&  < 2^{ - (k+1)} +  2^{-k}  +  |a_{k+1}| .
\end{split}
\end{align} 
          From \eqref{4.40}  and  our choice of  $ (a_k) $  we see for $  (\si_k) $  as in Lemma \ref{lem4.3} that  
          $\ds \lim_{r \to 0}   \ti \si ( r e^{i \he} ) $  does not exist for  almost  every  $  \he  \in [-\pi, \pi]$  while 
          if $ (\si_k) $  is as in Lemma \ref{lem4.4},   $\ds \lim_{r \to 0}   \ti \si ( r e^{i \he} ) = 0 $  almost everywhere. 
                   Moreover from  boundedness of  $ (\si_k)   $  and  the maximum principle for $p$-harmonic  functions we deduce that 
          $  \ti \si $ is bounded in  Lemma  \ref{lem4.3} or  \ref{lem4.4},  as well as  non-negative in Lemma  \ref{lem4.4}.     
          To conclude  the proof of  Theorem \ref{thmB},  put  $  \ti \si =  \hat u $  and  $ \ti \si = \hat v $ if   
          Lemma  \ref{lem4.3} and  Lemma  \ref{lem4.4}, respectively,  was used to construct $ \ti \si. $  
\end{proof} 

\section{Proof of Theorem \ref{thmC}} 
\label{sec5}  
\setcounter{equation}{0} 
\setcounter{theorem}{0}

In this section we  use Theorem \ref{thmA} to  prove Theorem \ref{thmC}. Except for minor glitches 
we shall  essentially copy  the proof in  \cite{LMW}. Once again     Lemma 4.1  plays an  
important role in the estimates.     Let $  (N_j)_1^{\infty} $   be a sequence  
of positive integers  with  $ N_1 = 1$  and  with $ (N_j)_2^{\infty}$  to be chosen later
 in order to satisfy several conditions.  For the moment  we assume only that 
    $ N_{j+1}/N_j  \geq  2.$    Let   $  \Ph_j  $  for $  j = 1, 2, \dots $ be the  $p$-harmonic function in $ B (0, 1) $  
    with  period $ 2 \pi/N_j $  constructed in   Theorem \ref{thmB}    with  $  \Ph_j = V. $     
    Following  \cite{LMW},  we assume as we may,   that 
\begin{align}  
\label{5.1}    \| \Ph_j \|_{\infty}  \leq  1/2   \quad \mbox{and} \quad   \int_{-\pi}^{\pi}  \log ( 1 + \Ph_j ) ( e^{i\he}) d \he    \geq  c_2^{-1}
\end{align}
      for $ j = 1, 2, \dots, $  where $ c_2 \geq  1  $  depends only on $p. $   
      Indeed otherwise,  we replace  $ \Ph_j  $  by  $ \ti \Ph_j = c^{-1}    \Ph_j $ 
      and observe from Theorem \ref{thmA}, elementary facts about power series that for  $ c > > c_1,  $    
\[  
\int_{-\pi}^{\pi}  \log (  1 +  \ti  \Ph_j (e^{i \he} ) ) d \he   \geq   \int_{-\pi}^\pi  \ti \Ph_j ( e^{i\he} )  d\he   -   2\pi  ( c_1/c)^2  \geq  ( 2 c_1 c )^{-1}. 
\]   
Thus we assume \eqref{5.1} holds. 
   We  claim that there exists a positive integer $ \kappa  > > 1 $   and a 
   positive constant $  C  = C ( \kappa)  >  1 $  such that     
   for $ j = 1, 2, \dots, $ 
\begin{align}
    \label{5.2}       
    \sum_{l = 1}^{\kappa} a_{l j}  \geq C^{-1}  \quad \mbox{and}\quad   \prod_{l=1}^{\kappa} ( 1 + a_{lj} )  > 1 +  C^{-1} 
\end{align}
where    for $ l = 1,  \dots, \kappa$,
\[
a_{lj} : =   \min \{ \Ph_j ( e^{i \he/N_j } )  : \, \, \he  \in  [ -  \pi   +  \frac{(2l - 2) \pi}{ \kappa },    -  \pi   +  \frac{2l  \pi}{ \kappa }]\}.
\]
   To prove  \eqref{5.2}, let  $\ph_j ( e^{i\he} )  =  \Ph_j  ( e^{i\he/N_j } )$ for $\he \in \re. $  
Then from Theorem  \ref{thmA} we see that   $ \ph_j $ is continuous  and $ 2\pi $ periodic on $ \re $ with  
$  \| \ph_j   \breve \|   \leq c_1,  $  where $ c_1 $ depends only on $p.$    Using these facts we get 
   
 \[  
 2 \pi    \kappa^{-1}    \sum_{l=1}^{\kappa} a_{lj} \geq   \int_{-\pi}^{\pi}  \ph_j ( e^{i\he} ) d \he  
-  \hat c \,  \kappa^{-1}   \geq  \frac{1}{2 c_1} > 0   
\] 
for $ \kappa $ large enough  thanks to  \eqref{1.3} $(c)$ .   Likewise,   from Theorem \ref{thmA} and \eqref{5.1}  it follows that
 \[  
 2 \pi       \sum_{l=1}^{\kappa} \log (  1  +  a_{lj} )  \geq   \kappa  \int_{-\pi}^{\pi} \log ( 1 +  \ph_j ) ( e^{i\he} ) d \he  -   c'  \,    >  \frac{\kappa }{2 c_2}   \,   
 \]  
 for  $ \kappa $ large enough where $ c_2 $   depends only on $p.$      
 Dividing    this inequality  by $ 2 \pi $  and exponentiating    we get the 
 second inequality in  \eqref{5.2}.    Hence  \eqref{5.2} is valid. 
From   \eqref{5.2}     we deduce for  $ j = 1, 2, \dots, $  the existence of  $  \La$ and $\ti N_0$   so that  
\begin{align} 
\label{5.5}  
\begin{split}
&  (a)  \hs{.2in} {\ds   1 <   \La <    ( 1 +  C^{-1} )^{ 1/ \kappa}  < \prod_{l=1}^{\kappa} ( 1 + a_{lj} )^{1/\kappa}}, \\
&  (b)  \hs{.2in}  3^{ - \ti N_0}    <  {\ds \min_{j}   \left[ 1 + \max_{1\leq l \leq \kappa} a_{lj}  -  \La,\,     ( c_1 \kappa)^{-1}     \right]}.
\end{split}
\end{align}
Fix $ \kappa $  subject to the above  requirements.  For $  \he \in \re$ and $k=1, \ldots, \kappa$, we let      
\begin{align}
 \label{5.6} 
 \begin{split}
q_1^k ( e^{i\he} )    :=  \Ph_1 ( - e^{ i  ( \he +  2k\pi/\kappa) })\quad \mbox{and} \quad  f_1^k ( e^{i\he}  )  :=  1  +  q_1^k ( e^{i\he} ).  
\end{split} 
 \end{align}
 Moreover,  for $\he \in \re$,  $ j = 2, 3, \dots, $  and  $  k = 1, \dots, \kappa$  set    
\begin{align} 
\label{5.7} 
q_j^k ( e^{i\he} )   : =  \Ph_j ( - e^{ i  ( \he +  2k\pi/ \kappa) }  )  \quad \mbox{and}\quad  f_j^k ( e^{i\he}  )  :=  ( 1  +  q_j^k ( e^{i\he} ) )  f_{j-1}^k  (e^{i\he} ).  
\end{align}
   Observe from  \eqref{5.1}, \eqref{5.5}, \eqref{5.7} that         
\begin{align} 
\label{5.8}    
\prod_{k=1}^{\kappa}   f_j^k ( e^{i\he} )   =   \prod_{l=1}^j \prod_{k=1}^{\kappa}  \left( 1  + \Ph_l  ( -  e^{ i ( \he +  2k\pi/\kappa) } ) \right)  > \La^{\kappa j}.
\end{align}   
Let 
\begin{align} 
\label{5.9} 
E_k :=  \{ e^{ i \he}  \in  \ar B ( 0, 1)  :  f_j^k ( e^{i\he} ) >  \La^j  \mbox{ for infinitely many }  j    \}.  
\end{align}
 From \eqref{5.8} we see that 
\begin{align}  
\label{5.10}    
\bigcup_{k = 1 }^{\kappa}  E_k    =  \ar B ( 0, 1). 
\end{align}
From  \eqref{5.10} we conclude that to  finish the proof of Theorem  \ref{thmC}  it suffices to show $k = 1, \dots, \kappa$ that 
\begin{align}    
\label{5.11}  
\om_p (0,  E_k ) = 0, \quad  \om_p (0,  \ar B ( 0, 1) \sem E_k ) = 1, \quad \mbox{and} \quad | \ar B ( 0, 1 ) \sem E_k | = 0  
\end{align}
where $ \om_p $ is defined after \eqref{1.4}.  To do this we   use   Lemma  \ref{lem4.1}   
and  an inductive type argument to  choose $  (N_j)_2^{\infty} . $  First   we  require that  
$ N_1  = 1$  and   $   N_{j+1}/N_j $   is divisible  by  $ \kappa $  for $ j =1, 2, \dots $. 
Second  for fixed $k$   and $ j = 1, 2, \dots $   we apply  Lemma \ref{lem4.1}   
with   $  f = g =  f^k_j$ and  $q = q^k_{j+1}$  From \eqref{5.1}  and Theorem \ref{thmB} 
    we see that     $ \| q^k_j \|_{\infty}  \leq 1/2,  $ and $ \| q^k_j  \breve \| \leq c _1 N_j $ for  $ j = 1, 2, \dots $  Thus, 
\begin{align}
\label{5.12}     
2^{-j}  \leq  \| f^k_j \|_{\infty} \leq (3/2)^j  \quad \mbox{and} \quad  \| f^k_j  \breve \|  \leq  c_1 \, 2^j   N_j .
\end{align}
     Let   $   M_j  =  c_1  4^j  N_j $  and  $  \ep  =  \ep_j  =  3^{-j - 1}. $   
     Then  there exists   small $ A_j = A_j (p,  \ep_j, M_j ), $   and large  $ \nu_0  (p, \ep_j,  N_j )  $ such that  if  $ N_{j+1} > \nu_0, $ then  
\begin{align}
\label{5.13}    
|  \hat f^k_{j+1} ( r e^{ i \he} )  -  f^k_{j} ( e^{i\he} ) ( q_{j+1} ( r e^{i\he} )  + 1 )  | < 3^{- (j+1) } \quad \mbox{for} \, \, 1 - A_j  N_{j+1}^{- \al} < r  < 1
\end{align}
and
\begin{align} 
\label{5.14}
|  \hat f_{j+1} (r e^{i \he} )    -   \hat  f_j (r e^{i\he} |   <  3^{- (j+1)} \quad \mbox{for}\, \,   r  \leq   1      -  A_j  N_{j+1}^{-\al} .
\end{align}
   Now   using    \eqref{2.3a}   as in the  derivation of  \eqref{4.4}  from  
     \eqref{4.3}   we see that   we may also  assume 
\begin{align}  
\label{5.15}     
|  \hat f_{j} (r e^{i \he} )    -   \hat  f_j ( e^{i\he} |   <    
    3^{- (j+1)} \quad \mbox{for}\, \,   r   \geq    1      -  A_j  N_{j+1}^{-\al}.
    \end{align}
  Finally,  we  may  choose  $  (A_j)$ and $( N_{j})  $      so that              
\begin{align}  
\label{5.16} 100 \, N_{j+1}^{-1}   <   t_j   =  A_j N_{j+1}^{-\al}  <  (c_1 N_j 6^{j+1} )^{-1} \quad \mbox{and}\quad      t_{j+1}  <   \frac{t_j}{  \kappa}
\end{align}
for $j=1,2, \ldots,$.   From  \eqref{5.14} and \eqref{5.16},    we deduce    for  $ m > j, $  a positive integer, and for $ k =1,2, \dots, \kappa, $  that  
\begin{align} 
\label{5.17}     
|  \hat f^k_m (r e^{i \he} )    -   \hat  f^k_j (r e^{i\he} |   \leq 3^{- j} \quad \mbox{for}\, \,   r  \leq   1     -  t_j  . 
\end{align}
  From \eqref{5.17} and Lemmas \ref{lem2.1} - \ref{lem2.3}  we  obtain that  
   $ \hat f^k_j$ and  $\nabla  \hat f^k_j$  converge uniformly as  $ j \to \infty $  to  a   locally  $p$-harmonic  $ \hat f^k, \, \nabla \hat f^k,  $   on compact subsets of  $ B (0, 1)$ satisfying  \eqref{5.17} with $ \hat f^k_m $ replaced by  $ \hat f^k.$    Also  from  \eqref{5.13},    \eqref{5.1},  and  \eqref{5.17}  with $ j $ replaced by $ j + 1, $    it follows that 
\begin{align} 
\label{5.18}       
\hat f^k_m (r e^{i \he} )   \geq   \frac{1}{2}  \hat  f^k_j (r e^{i\he} )  -  3^{- j} \quad \mbox{for}\, \,    1 - t_j  \leq   r  \leq   1   -  t_{j+1}   . 
\end{align}    
 Next   for  fixed  $ k$, $1 \leq  k  \leq \kappa, $  let   $  G_j^k  =  \{ e^{i\he}:  f^k_j ( e^{i\he} ) >  \La^j \} . $  
 Then   
 \begin{align}  
 \label{5.19}  
 E_k   =  \bigcap_{n=1}^{\infty}    \left( \bigcup_{j=n}^{\infty} G^k_j  \right) \mbox{ where $ E_k $ is  as in   \eqref{5.9}.} 
\end{align} 
   By monotonicity of  $p$-harmonic measure  it  suffices  to  show that 
\begin{align}  
\label{5.20}  
\om_p  \left( 0,   \bigcup_{j=n}^{\infty} G^k_j  \right)    \leq  \ti C  \La^{-n} \quad \mbox{for}\, \, n = 1, 2, \dots  
\end{align}
 where $ \ti C  \geq 1 $  does not depend on $n.$   
Moreover, from   Theorems 11.3-11.4  and  Corollary  11.5 in  \cite{HKM}  applied to   $   \om_p  \left(  0, \ar B ( 0, 1 )  \sem \bigcup_{j = n}^N  G^k_j \right)$ we  see  that  
   \[    
\lim_{N \to \infty}  \om_p \left ( 0,  \bigcup_{j = n}^N  G^k_j  \right) = \om_p  \left( 0,   \bigcup_{j=n}^{\infty} G^k_j  \right).
\]   
Therefore,  instead of  proving  \eqref{5.20},   we  need only show that 
\begin{align} 
\label{5.21}  
\om_p  \left( 0, \bigcup_{j=n}^{ N } G^k_j  \right)    \leq \ti C  \La^{-n} \quad   \mbox{for}\, \,  N  > n    
\end{align}
 in order to  conclude    
       that $ \om_p ( 0, E_k ) = 0. $ This  conclusion  and  Theorem   11.4 in  \cite{HKM}  then yield  
  $  \om_p ( 0, \ar B ( 0, 1)\sem E_k ) = 1 $      for   $ k = 1, 2, \dots, \kappa .$   
  
  To prove  \eqref{5.21} we temporarily drop the $k$ and write  $ f_j, G_j $ for $  f_j^k, G_j^k. $  Let 
  \[ 
  H_j :=   \bigcup \{  I \subset   \re  : I  \mbox{ is a closed interval of  length }  t_j,    \max_{\he \in I} f_j (e^{i\he}) 
\geq \La^{j}  -  3^{-j - 1} \} 
\]  
and let  $\mathring H$  denote the interior of $ H $ relative to $ \re.$    Clearly, 
\begin{align}  
\label{5.22}     f_j  (e^{i\he})   <  \La^j   -  3^{-(j+1)} \quad \mbox {if}\, \,  \he  \in H_j  \sem \mathring H_j. 
\end{align}
   Hence   
$ \{ \he : e^{i\he} \in   \bar  G_j \}  \subset \mathring H_j. $  From   \eqref{5.12},  \eqref{5.16} we see that 
\begin{align} 
\label{5.23}  | f_j (e^{i\he_1} )   -  f_j (e^{i\he_2} )  |  \leq c_1  2^{j}  t_j  \leq 3^{- j} 6^{-1} \mbox{ if }  
     | \he_1 - \he_2 |  \leq t_j .
\end{align}     
Thus  
\begin{align} 
\label{5.24}  
\min_{\he \in H_j}  f_j (e^{i\he} ) \geq \La^{-j}  -  3^{-j} 2^{-1}.
\end{align}
  Let 
\[
T_j  =  \bigcup \, \{ I \times [0, t_j] : I \in H_j \}  \subset   \mathbb{\bar R}_+^2 \mbox{ for }  j  = 1, 2, \dots  
\]  
Using  \eqref{5.24}, \eqref{5.15}, \eqref{5.16},  we conclude   that 
\begin{align}  
\label{5.25} 
\hat f_j ( r e^{i \he} )  >   \La^{j}  - 3^{-j}   \quad \mbox{if}\, \,   ( \he, 1 - r ) \in T_j .
\end{align}
     At this point the authors in \cite{LMW} note that  if it  were true  that   
     \[  
     \hat f_N ( re^{i\he} )  >  \bar  C^{-1} \La^j   \mbox{ for $ ( \he, 1 - r ) $  in the  closure of  $ \rn{2}_+  \cap    \ar  T_j    $  for    $ N  \geq j \geq n >   \ti N_0, $ }  
     \]   then it  would follow from the  boundary  maximum principle for  $p$-harmonic  functions  applied to   
    $ \bar   C  \La^{- n}  f_N $  in   
\[   
B ( 0, 1) \sem \{ r e^{i\he} :  ( \he,   1 - r) \in  \bigcup_{j=n}^N  \bar T_j \} 
\]   
and  convergence of     
     $ (\hat f_j) $  to  $  \hat f $   that \eqref{5.21} is valid.       
     Unfortunately, this  inequality need not hold so the authors modify 
     the components of  $ T_j $  as follows.      Observe that $ T_j $  
     has a finite number of  components having a non-empty intersection with  $ [- \pi, \pi ]. $   
     If $ Q = [a, b] \times [0, t_j] $ is one of  these components then  
\begin{align} 
\label{5.26}
 f_j ( e^{ia} )  , f_j ( e^{ib} )  <  \La^j  - 3^{- j - 1} \mbox{ thanks to \eqref{5.22}.}
\end{align}
     If   $ {\ds \max_{ \he \in [a, b]}}  f_j ( e^{i\he} )  \leq  \La^j, $  
     then from \eqref{5.26} and the definition of  $ G_j $    we deduce that   
\[
\{e^{i\he} : \he \in [a, b] \} \cap \bar G_j = \es
\]  so in this case  put  $ Q^* = \es.$    
Otherwise 
${\ds  \max_{ \he \in [a, b]}}  f_j ( e^{i\he} )   >  \La^j, $  and 
from    \eqref{5.12}, \eqref{5.16}, \eqref{5.26},\eqref{5.25},   
we see that if $ I_j^Q =   [a, a + t_j ]$  and  $  J_j^Q  = [b - t_j, b ], $  then    
\begin{align} 
\label{5.27}  
\La^j  -  3^{- j }         <  \,  f_j ( e^{i\he} )  \,  <  \La^j  - 3^{- j - 2} \quad \mbox{on}\, \,   I_j^Q \cup J_j^Q. 
\end{align}
  We note from  \eqref{5.16}  that  $ N_{j+1} t_j > 100$   and  $ q_{j+1} (e^{i\he} )   $ is 
$  \frac{2\pi}{N_{j+1}} $  periodic in  $ \he $   so from the definition of  $ (a_{l(j+1)})$   
and \eqref{5.16}  we can  find intervals 
 $ I_{j+1}^Q$ and $J_{j+1}^Q $  with 
\begin{align} 
\label{5.28}   
I_{j+1}^Q \subset  I_{j}^Q ,    J_{j+1}^Q \subset  J_{j}^Q,  \mbox{ and }  \max_{I_{j+1}} q_{j+1},   \,    \max_{J_{j+1}} q_{j+1}  \geq   \max_{1\leq l \leq \kappa} \,  a_{l(j+1)} .  
\end{align}
Moreover  $  I_{j+1}$ and $J_{j+1} $  each have length  $ t_{j+1}. $ Then  from \eqref{5.5} $(b)$ 
and \eqref{5.28} we get  for  $ L =  I_{j+1}$  or $ J_{j+1} $ 
 and $  j  > \ti N_0 $ that       
\begin{align} 
\label{5.29}  
\max_{L} f_{j+1} ( e^{i\he} )  \geq ( 1 +  \max_{1\leq l \leq \kappa} a_{l(j+1)} ) ( \La^j  - 3^{-j} )   >   \La^{  j + 1}. 
\end{align}
  From  \eqref{5.29}, \eqref{5.12}, \eqref{5.16},  with  $ j $ replaced by  $ j  + 1, $ 
 we  deduce  that 
\begin{align}
\label{5.29a}   
\min_{L} f_{j+1} ( e^{i\he} ) \geq \La^{j+1}  - 3^{- j - 2}.
\end{align}
   We  can   now  argue by induction    to get   nested closed intervals 
   $ ( I^Q_l )_j^{\infty}$ and $( J^Q_l )_j^{\infty}, $  for which  
 $  I^Q_l $ and $ J^Q_l$   have  length $ t_l$   and   \eqref{5.29},   \eqref{5.29a}, 
 are valid  with  $ j + 1 $ replaced by $ l. $   Then     
\begin{align}
\label{5.30}   
a  <  a^*  :=   \bigcap_{l=j}^{\infty} I_l^{Q} \quad \mbox{and}\quad  b^* :=      \bigcap_{l=j}^{\infty}  J_l^Q  < b.
\end{align}
 Set   $ Q^*  = [a^*, b^*] \times [0, t_j] $  and 
 \[  
 T_j^* =  \bigcup \{ Q^* :  Q \mbox{ is a component of $ T_j$} \}.   
 \]    
 Then by  construction and \eqref{5.27}
 \[ 
 \bar G_j \subset \mathring H_j^*  \subset  H_j^* = \ar T_j^* \cap \re .    
 \] 
 Finally the  authors show   
  \begin{align} 
  \label{5.31}  
  \hat f_N ( re^{i\he}) > \frac{1}{3} \La^j  \mbox{ for $ ( \he, 1 - r)$ in }  \ar T_j^*  \sem  \mathring H_j^* \mbox{ and } N \geq j. 
  \end{align}
 For $ N = j $ this inequality is implied by  \eqref{5.25} while if $t_{j+1}  \leq  1 -  r   \leq  t_{j} $ 
 we see from 
  \eqref{5.18} and  \eqref{5.25} that \eqref{5.31} is valid  for 
  $ ( \he, 1 - r ) $  in   $ \ar T^*_j  \cap [t_{j+1}  \leq 1 -  r \leq  t_{j}] .$  
   The  only remaining segments  of  $ \ar T_j^* \cap \rn{2}_+ $      are of the 
   form  $  \{a^*\}  \times  [0, t_{j+1}],   \{b^*\}  \times  [0, t_{j+1}], $ where $ a^*, b^* $  are as in  \eqref{5.30}.                     
 If  $  (\he, 1 - r ) \in  \{a^*\} \times  [ t_{l+1}, t_{l}] $  or  $  \{b^*\} \times  [ t_{l +1}, t_{l}] $ for $ j+1 \leq l < N $  
 we can   use    \eqref{5.18} with $m = N,  j = l, $ \eqref{5.15} with $ j = l, $  
 and  \eqref{5.29a} with $ j + 1  = l$  to get  that \eqref{5.31} is valid on
   $ \ar T^*_j  \cap [  t_{l+1}  \leq 1 - r \leq   t_{l}] .$     
   If  $  (\he, 1 - r ) \in  \{a^*\} \times  [ 0, t_{N}] $  or  $  \{b^*\} \times  [ 0, t_{N}] $   
    then  from  \eqref{5.15}  with $ j = N $  and  \eqref{5.29a} with  
 $ j + 1  = N,  $  we   obtain \eqref{5.31}  on    
 $ \ar T^*_j  \cap (0 <  1 - r \leq   t_{N}] .$   
 Thus \eqref{5.31} is valid and from the discussions after  \eqref{5.25},   
 \eqref{5.21}, we  conclude   $  \om_p (0,  E_k ) = 0 $ for $ 1 \leq k \leq \kappa. $   
 
 It remains  to   prove  $ | \ar B (0, 1) \sem E_k |  = 0$ for $1 \leq k  \leq  \kappa. $  
 To do this,  for  $ j = 1, 2, \dots,$  let  $  \tau_j (e^{i\he} )  = \log( 1 + a_{lj})    $ 
 when $  \he \in   [ - \pi + \frac{(2l-2) \pi }{\kappa}, - \pi 
+  \frac{2l \pi}{\kappa}) $ and $ 1 \leq l \leq \kappa.$  We  regard  $ \tau_j $  as  a  $ 2 \pi $ 
periodic function on $ \re. $   For $ \he \in \re$,  $ j = 1, 2, \dots,$ and $k = 1, \dots, \kappa, $  let  
\[  
h^k_j  (e^{i\he} )  =  \tau_j (-  e^{i  (N_j \he + 2k\pi/\kappa)})   - m_j   \quad  \mbox{ where } \quad m_j =  \frac{2\pi}{\kappa} \sum_{l=1}^{\kappa} \log ( 1 + a_{l j}).  
    \]        
    Then for fixed $k$,    $h^k_{j+1}  $ is  $ 2 \pi/N_{j+1}  $  periodic  and has  average 0   on intervals  where  $ h^k_j $ is constant  since  
$ N_{j+1}/N_j $ is divisible by $ \kappa. $   Thus  for fixed $k, $   the  
functions        $ h_j^k ( e^{i\he} )$  are orthogonal
 in  $ L_2 ( \ar B (0, 1) ) $  and also  uniformly bounded for $j = 1, 2, \dots$.   Using this fact one can 
show  (see  page 182 in  \cite{KW})  that          
\begin{align} 
\label{5.32}    
\sum_{l = 1}^j   h_l^k   =  O (j^{3/4})  \mbox{ for almost every $ e^{i\he} \in \ar B (0, 1) $ }
\end{align}
 with respect to Lebesgue measure on $  \ar B ( 0, 1). $ 
  Since   
  \[  
  \log f_j^k  \geq   \sum_{l=1}^j  ( h_l^k  +   m_l ) 
  \] 
  it follows from \eqref{5.5} $(a)$, \eqref{5.32} that for 
  almost every $ e^{i\he} \in \ar B ( 0, 1) $ there exists $ j_0 ( \he ) $ 
  such that  for $ j \geq j_0, $        
$ f_j^k ( e^{i\he} ) >  \La^j $ .  From the definition of $ E_k $  
we arrive at  $  | \ar B ( 0, 1) \sem  E_k | = 0 $ for $ 1 \leq j \leq k. $

\section{Closing Remarks} 
\label{sec6}  
\setcounter{equation}{0}  
\setcounter{theorem}{0}     
We  note  that  in \cite[section 4]{LMW},  the  authors  discuss  some  interesting  open  questions   for  $p$-harmonic measure.   
Theorems \ref{thmA}, \ref{thmB}, \ref{thmC}   were  inspired  by  these  questions.     
One  natural  question  is  to  what  extent     Theorem \ref{thm1.1}  or  Theorem \ref{thmB}
has  an  analogue   in other domains?   For     example,  can  one   prove  similar  
theorems  in     the  unit  ball,  say   $ \breve B ,$   of    
$  \rn{n}  =  \{ x =  (x_1, \dots,  x_n) :   x_i \in \re,  1 \leq i \leq n\}, $   $   n  \geq 3,  $  
when $ 1 < p < \infty,  p  \not = n$?     For  $ p  = n, $    
one  can map $  \rn{n}_+ =  \{ x \in \rn{n} : x_n > 0 \} $ by way  of  a  linear  fractional  transformation,   
conformally    onto   $ \breve B $  and  use invariance of  
the  $ n$-Laplacian under conformal  mappings  to  
conclude  that   the conclusion in  Theorem \ref{thm1.1}   extends to    $  \breve  B.  $      
Theorems \ref{thmB} and \ref{thmC}   generalize  to     $  B (0, 1 ) \times  \rn{n-2},   n   \geq 3,   $  
by  adding  $ n - 2 $  dummy  variables.  
We note that for  $ p >2, $  the Martin function  in  $ \rn{n}_+  =   \{  x \in \rn{n} : x_n > 0 \}  $   relative to 
 0 is  homogeneous  of  degree   $  - \la $  where  $ 0 < \la  <  N - 1 $  as follows from  
 Theorem 1.1  in 
   \cite{LMTW}  (see also  \cite{DB1}).    Using  this  fact  one  can   
   construct examples in $ \rn{n}_+ $ satisfying  Theorem \ref{thm1.2}  similar to 
   the hands on examples  constructed in  $ \rn{2}_+. $  
     
Another  interesting question  is   whether   the   set in  Theorem \ref{thm1.1}  or  
Theorem \ref{thmB}  where radial  limits  exist  can have  Hausdorff  dimension  $ <  1?$    
 This  set  has  dimension   $ \geq    a = a (p)  > 0 $   thanks  to  work  of   \cite{MW}  and 
\cite{FGMS}.  

 Also   an   interesting question to us is whether Theorem \ref{thm1.1} or  Theorem \ref{thmB}  have    
 analogues for  solutions to  more  general  PDE  of  $p$-Laplace type.  
 To  give  an  example,   given $  p$,  $1 < p < \infty, $    suppose  
 $ f : \rn{n} \sem  \{0\}    \to  (0, \infty) $  has  continuous  third  partials on $ \rn{n} \sem \{0\} $  with 
\begin{align}
\label{6.1}
\begin{split}
(a)&\, \,   f  (  t   \eta )  =  t^p   f ( \eta ) \quad \mbox{when}\, \,   t  > 0 \, \, \mbox{and}\, \,  \eta \in  \rn{n},
\\
(b)& \, \,   \mbox{There exists $  \ti a_1 = \ti a_1 (p)   \geq 1 $  such that if  $ \eta, \xi   \in  \rn{n} \sem \{0\},$  then } \\
&  \hs{.44in}      \ti a_1^{-1}  \,   | \xi  |^2     | \eta |^{p-2}  \,  \leq   \,  
\sum_{i, j = 1}^n  \frac{ \ar^2  f }{ \ar \eta_i \ar \eta_j}  ( \eta )   \,  \xi_i   \,  \xi_j     \leq             \ti a_1  \,  | \xi  |^2     | \eta |^{p-2}.    
\end{split}
\end{align}
 Put    $  \mathcal{A} = \nabla f  $   for fixed $ p > 1. $ Given an   open set 
 $ O  $ we say that $v  $ is $  \mathcal{A}$-harmonic in $ O $ provided $ v \in W^ {1,p} ( G ) $ for each open $ G $ with  $ \bar G \subset O $ and   
	\begin{align}
	\label{6.2}
		\int \lan    \mathcal{A}
(\nabla v(y)), \nabla \he ( y ) \ran \, dy = 0 \quad \mbox{whenever} \, \,\he \in W^{1, p}_0 ( G ).
			\end{align} 
Note that if   $ f (\eta  )  =  | \eta |^p$  in \eqref{6.1} then $ v  $  as in \eqref{6.2} is $p$-harmonic in  $ O.$      
Also observe that if  $ v $  is $ \mathcal{A} = \nabla f $-harmonic   
in  $ \rn{2}_+ $ then   $  \ti  v  (  z  )  =  v ( N z  + z_0 )  $  is also
  $  \mathcal{A} = \nabla f $  harmonic in $ \rn{2}_+ $  for  $ z, z_0  \in \rn{2}_+  $ and $ N \in \re.$     
  As  mentioned  earlier    Wolff made   important use of    similar    translation,  
  dilation  invariance  for   $ p $-harmonic functions.   Thus we believe  Theorems \ref{thm1.1}  
  stands a good  chance of  generalizing to the $  \mathcal{A} $-harmonic  setting.  
  On the other  hand     we  made  important use of  rotational  invariance of   
  $ p $-harmonic  functions in our  proof  of  Theorem \ref{thmB}.   Since  this  invariance
   is  not  true in general  for   $ \mathcal{A} = \nabla f $-harmonic  functions
     on  $  B (0, 1 ), $  an  extension of  Theorem  \ref{thmB}  to   
the $ \mathcal{A}$-harmonic setting would  require new techniques.

Finally, we  note that   in   a  bounded  domain  $ D \subset \rn{n}  $ with $ 0 \in D $   
and  for  $ p  = 2 $   one can  show  $  \om_2 (  \cdot ) $ (known as harmonic measure)  is  a  positive  Borel  measure  on  $ \ar D$,  
associated with  the  Green's function of $ D $  having  a  pole  at    $ 0. $   
This notion of  harmonic measure  led  the  authors  with  various  co-authors   in   \cite{BL,LNP,L2,LNV,ALV} 
to    study     the  Hausdorff  dimension of     a positive Borel  measure with support in  $  \ar D,  $  
associated with a  positive   $ p$-harmonic  function  defined  in   $  D \cap N $ and   
with  continuous  boundary value  0 on $ \ar  D. $  Here $ N $  is an open  neighbourhood of  $  \ar  D.$     
Moreover,  many of   the  dimension results we obtained for these ``$p$  harmonic  measures''  
in the above papers  were  also  shown in  \cite{Akman,ALV1}  to  hold  for  
the  positive  Borel measures  associated with  positive
  $  \mathcal{A} =   \nabla  f $-harmonic functions in $ D \cap N, $ vanishing on  $ \ar D. $     

\nocite{*}
\bibliographystyle{amsbeta}
\bibliography{jobname}    

     \end{document}